\documentclass[10pt]{article}
\usepackage[margin=2cm]{geometry}

\usepackage{amsmath,amsfonts,amsthm,amssymb,amscd,color,xcolor,mathrsfs,verbatim,microtype}
\usepackage{graphicx,eurosym}
\usepackage{hyperref}
\usepackage{mathtools}
\usepackage{appendix}

\usepackage[applemac]{inputenc}

\usepackage[cyr]{aeguill}

\colorlet{darkblue}{blue!50!black}

\hypersetup{
    colorlinks,%
    citecolor=blue,%
    filecolor=red,%
    linkcolor=darkblue,%
    urlcolor=blue,%
    pdfnewwindow=true,%
    pdfstartview={FitH}
}

\usepackage{graphicx,amscd,mathrsfs,wrapfig,mathrsfs,lipsum}
\usepackage{eufrak}
\usepackage{float}
\usepackage{tikz}
\usepackage{multicol}
\usepackage{caption}
\usetikzlibrary{arrows}
\usepackage{capt-of}

\colorlet{darkblue}{blue!50!black}

\newcommand{\p}{\partial}
\newcommand{\e}{\varepsilon}

\newcommand{\R}{{\mathbb R}}
\newcommand{\Z}{{\mathbb Z}}

\newcommand{\T}{{\mathbb T}}

\newcommand{\N}{{\mathbb N}}

\newcommand{\ty}{\infty}

\newcommand{\de}{\delta}

\DeclareMathOperator{\sign}{sign}

\newcommand{\TT}{{\cal T}}

\newcommand{\lag}{\langle}
\newcommand{\rag}{\rangle}

\newcommand{\dd}{{\textup d}}

\newcommand{\B}{{\mathbb B}}

\newcommand{\dist}{\mathop{\rm dist}\nolimits}

\theoremstyle{plain}

\newtheorem*{mta}{Main Theorem~\hypertarget{A}{A}}
\newtheorem*{mtb}{Main Theorem~\hypertarget{B}{B}}

\newcommand{\norm}[1]{\left\lVert #1 \right\rVert}
\newcommand{\tT}{\tilde}

\newtheorem*{lemma*}{Lemma}
\newtheorem{theorem}{Theorem}[section]

\newtheorem{proposition}[theorem]{Proposition}
\newtheorem{corollary}[theorem]{Corollary}
\theoremstyle{definition}
\newtheorem{definition}[theorem]{Definition}

\theoremstyle{remark}

\newtheorem{remark}[theorem]{Remark}
\newtheorem{example}[theorem]{Example}
\newtheorem*{Assumption I}{Assumption I}
\newtheorem*{Assumption II}{Assumption II}

\numberwithin{equation}{section}

\title{
On the small-time bilinear control of a nonlinear heat equation: global approximate controllability and exact controllability to trajectories

}
 \author{Alessandro Duca,\footnote{Universit\'e de Lorraine, CNRS, INRIA, IECL, F-54000 Nancy, France;  e-mail: alessandro.duca@inria.fr}\and{Eugenio Pozzoli \footnote{Univ Rennes, CNRS, IRMAR - UMR 6625, F-35000 Rennes, France; e-mail: eugenio.pozzoli@univ-rennes.fr}}\and{Cristina Urbani\footnote{Dipartimento di Ingegneria e Scienze, Universitas Mercatorum, Piazza Mattei 10, 00186, Roma, Italy; e-mail: cristina.urbani@unimercatorum.it}}}

\begin{document} 

 \maketitle

\begin{abstract}
    In this work we analyse the small-time reachability properties of a nonlinear parabolic equation, by means of a bilinear control, posed on a torus of arbitrary dimension $d$. Under a saturation hypothesis on the control operators, we show the small-time approximate controllability between states sharing the same sign. Moreover, in the one-dimensional case $d=1$, we combine this property with a local exact controllability result, and prove the small-time exact controllability of any positive states towards the ground state of the evolution operator.

\vspace{2mm}

\vspace{2mm}

\end{abstract}

\section{Introduction}

In this paper we study controllability properties of the following Nonlinear Heat Equation on the $d$-dimensional torus $\T^d=\R^d/2\pi \Z^d$

\begin{equation}\label{0.1}\tag{NHE}\begin{split}
    \begin{cases}
 \p_t \psi(t,x)=\Big(\Delta -\kappa  \psi(t,x)^p +\lag u(t),Q(x) \rag\Big)\psi(t,x), \ \ \ \ \ \ \ \ x\in\T^d,\ t>0,\\
   \psi(t=0,\cdot)=\psi_0(\cdot),
    \end{cases}
\end{split}
 \end{equation}
where $d\in\N^*$, $p\in\N$,  $\kappa\in\R$, the operator $\Delta=\sum_{i=1}^d\frac{\partial^2}{\partial x_i^2}$ is the Laplacian and $Q=(Q_1,...,Q_{q},\mu_1,\mu_2):\T^d\to \R^{q+2}, q\in \N, q\geq 2d+1,$
is a fixed bounded measurable function. The set of potentials $(Q_1,\dots,Q_q)$ is used to establish global approximate controllability properties, as detailed in Section \ref{sec:approximate},  while the set of potentials $(\mu_1,\mu_2)$ ensures local exact controllability to trajectories, as described in Section \ref{sec:exact}.
The $\R^{q+2}$-valued function $u\in L^2_{loc}(\R^+,\R^{q+2})$ plays the role of a control. It means that $u$ is the function that can be chosen to steer the solution of the problem towards a desired state. Observe that our control depends only on time.

\smallskip

In numerous practical problems from chemistry, neurobiology, and life science, the evolution of a specific system is subject to a control which is not external, but rather a modification of the principal parameter of the evolution. In these cases, it can be appropriate to consider evolution equations in the presence of multiplicative controls. Such controls are called bilinear and have the form $\lag u(t),Q(x) \rag$, as in \eqref{0.1}, if only the time-dependent intensity $u$ can be tuned, while the space-dependent potential $Q$ is fixed. 

\smallskip

An example of a parabolic model with a dynamic governed by a multiplicative control is the distributed parameter control model studied by Lenhart and Bath in \cite{pop}. In such work, the authors consider wildlife damage management to control the population of diffusive small
mammal species such as beavers, raccoons, and muskrats. These small mammal
species often damage human interests, and it is important to study the possibility of controlling their dispersal behaviour. The migratory habit of such animals obviously presents an additional complication. For instance, their removal from any habitat can cause the attraction of other individuals from nearby lands and a consequent increase of the trapping cost. In \cite{pop}, the authors study the dynamics of the population density of one of these species, described by a control model incorporating dispersive dynamics and a multiplicative control that represents trapping. In such work, the evolution is modelled by an equation of the form
\begin{align}\label{introd}
    \p_t \varphi(t,x)=(\alpha \Delta + a - b\varphi(t,x)  + p(t,x)  )\varphi(t,x),\qquad t>0,
\end{align}
where $\varphi$ is the population density, $\alpha$ is constant, $a$ and $b$ are growth parameters and $p$ is the rate of trapping which is used as a control. When one can control the time-dependent intensity of the trapping but not its spatial distribution, it is possible to separate the variables of the function $p$ and write 
$$p(t,x)=u(t)\mu(x).$$
This choice leads to a new formulation for the evolution model where the actual control is the function $u$ and the equation for $a=0$, takes the form of \eqref{0.1}.

\subsection{Global approximate controllability on the $d$-dimensional torus}\label{sec:approximate}

Let us consider the following vector space:
\begin{equation*}\label{H0_as}
    \mathcal{H}_0:={\rm span}_{\R}  \{Q_1,\dots,Q_q\}.
\end{equation*} 
We also introduce a collection of $d$ elements of $\Z^d$ 
$$
\mathcal{K}=\Big\{\big(1,0,\dots,0\big),\big(0,1,\dots,0\big),\dots,\big(0,\dots,1,0\big),\big(1,\dots,1\big)\Big\}\subset \Z^d.
$$

 \begin{Assumption I} The potentials satisfy $Q_1,\dots,Q_q\in C^\infty(\T^d,\R)$, and
 \begin{equation}\label{ip-trigonotriche}\begin{split}
&\big\{1,\cos \lag k, x\rag ,\ \sin \lag k, x\rag\big\}_{k\in\mathcal{K}}\subset \mathcal{H}_0.\end{split}\end{equation}
\end{Assumption I}

In what follows, we denote by  
$\psi(t;\psi_0, u)$
the solution of \eqref{0.1} at time $t$, associated with the initial condition $\psi_0$ and the control $u$. Of course, it is intended that we look at such solutions only when the existence is ensured (see Proposition \ref{well-pos} below).
We now present the first main result of the paper: a small-time global approximate controllability property of \eqref{0.1} between states that share the same sign. 
\begin{mta}\label{mta}
Let $s\in\N^*$ be such that $s>d/2$, $\kappa\in \R,$ and $p\in \N$. Suppose that Assumption I is satisfied. Then \eqref{0.1} verifies the following small-time approximate controllability properties.
\begin{itemize}
\item[(i)] Let $\psi_0,\psi_1\in H^s(\T^d,\R)$ be such that ${\rm sign}(\psi_0)={\rm sign}(\psi_1)$. For any $\epsilon>0$ and $T>0$, there exist $\tau\in [0,T]$ and $(u_1,...,u_q)\in L^2((0,\tau),\R^q)$, such that 
$$\big\|\psi(\tau;\psi_0,u)-\psi_1\big\|_{L^2}<\epsilon,$$
with $u=(u_1,...,u_q,0,0)$.
\item[(ii)] Let $\psi_0,\psi_1\in H^s(\T^d,\R)$ be such that $\psi_0,\psi_1>0$ (or $\psi_0,\psi_1<0$).
For any $\epsilon>0$ and $T>0$, there exists $(u_1,...,u_q)\in L^2((0,T),\R^q)$, such that
$$\big\|\psi(T;\psi_0,u)-\psi_1\big\|_{H^s}<\epsilon,$$
with $u=(u_1,...,u_q,0,0)$.
\end{itemize}
\end{mta}
As stated in part $(i)$ of Main Theorem A, with our technique we are able to prove small-time  controllability between states sharing the same sign, approximately in $L^2$, but not in $H^s$ for $s>0$. As pointed out in part $(ii)$, for positive states, we can improve the small-time approximate controllability to be also valid in $H^s$. In particular, the $H^s$-approximate control result for $s>d/2$, stated in part $(ii)$, will be crucial to derive small-time global exact controllability of \eqref{0.1}, driving the solution from any positive $H^3$ state towards the constant state. Such a result is contained in Main Theorem B below.
Main Theorem A is a specific case of Theorem \ref{mta-gene} below, where approximate controllability is ensured with general control potentials $Q_1,\dots,Q_q$ satisfying specific saturation assumptions. 


\smallskip

To the best of our knowledge, the most recent works on approximate controllability of nonlinear parabolic equations via multiplicative controls are \cite{FloCanKah, FloNitTro}. In such works, the authors studied $1$-dimensional problems with globally Lipschitz continuous nonlinearities. They proved approximate controllability between states with the same number of sign changes when the time is sufficiently large. They considered controls whose space- and time-dependence can be chosen 
accordingly to the initial and final states. The novelties of Main Theorem A are thus the following:

\begin{itemize}
\item The approximate controllability is achieved at arbitrarily small times.


\item The approximate controllability holds on $\T^d$ for any $d\in \N^*$, that is, in arbitrary spatial dimensions.


\item The controls depend only on time, that is, the space-dependent potential $Q(x)$ is fixed.

\item Equation \eqref{0.1} exhibits a polynomial nonlinearity.
\end{itemize}

\subsection{Global exact controllability to the ground state solution for the $1$-dimensional problem}\label{sec:exact}

The second main result of our work is the exact controllability of \eqref{0.1} to the ground state solution in the $1$-dimensional case $d=1$. Let us consider the ordered eigenvalues $\{\lambda_k\}_{k\in\N}$ of the Laplacian $-\Delta$ (not counted with their multiplicity)
\begin{equation}\label{eigenv}
    \lambda_k=k^2,\qquad \forall\, k\in\N.
\end{equation}
Note that, except for the first one $\lambda_0=0$, all the eigenvalues are double. We denote by $\{c_0,c_k,s_k\}_{k\in\N}$ the corresponding orthonormal eigenfunctions of $-\Delta$
\begin{equation}\label{eigenf}
    c_0 = \frac{1}{\sqrt{2\pi}},\qquad c_k(x)= \frac{1}{\sqrt{\pi}}\cos( k x),\qquad s_k(x)= \frac{1}{\sqrt{\pi}}\sin( k x),\qquad \forall\, k\in\N^*,
\end{equation}
which form a Hilbert basis of $L^2(\T,\R)$. Note that $c_0$ represents the free evolution ($u=0$) of the linear ($\kappa=0$) heat equation \eqref{0.1} with initial condition $\psi_0=c_0$. Such a solution is usually called the \emph{ground state solution}. Henceforth, we shall also use the notation $\Phi:=c_0$. To study the exact controllability of \eqref{0.1}, we introduce the following additional assumption.
 
 \begin{Assumption II} The potentials satisfy $Q_1=1$, $\mu_1,\mu_2\in H^3(\T,\R)$ and
 \begin{equation}\begin{split}\label{hp_intro}
\lag \mu_1,c_0 \rag_{L^2}\neq 0,\ \ \ \ \ & \    \ \ \lag \mu_2,c_0 \rag_{L^2}=0,\\
     \exists\,b_1,q_1>0\,:\: \lambda_k^{q_1}\left|\lag \mu_1 ,c_k\rag_{L^2}\right|\geq b_1,\,\, &\text{ and }\,\, \lag \mu_1,s_k \rag_{L^2} = 0,\quad\forall\,k\in\N^*,\\
        \exists\,b_2,q_2>0\,:\: \lambda_k^{q_2}\left|\lag \mu_2 ,s_k\rag_{L^2}\right|\geq b_2,\,\, &\text{ and }\,\, \lag \mu_2,c_k \rag_{L^2}= 0,\quad\forall\,k\in\N^*.
\end{split}\end{equation}
\end{Assumption II}

We now state our second main result which ensures global small-time exact controllability of the 1-dimensional \eqref{0.1} to the ground state solution, starting from any positive state.

\begin{mtb}\label{mtb}
 Let $d=1$, $\kappa \geq 0$ and $p\in 2\N$. Suppose that Assumptions I and II are satisfied. Then, \eqref{0.1} is exactly controllable in $H^3(\T,\R)$ to the ground state solution $\Phi $ in any positive time from any positive state. More precisely, for any $T>0$ and $ \psi_0 \in H^3(\T,\R)$ such that $\psi_0>0$, there exists $u\in L^2((0,T),\R^{q+2})$, such that $$\psi(T;  \psi_0 ,u)= \Phi.$$
Analogously, for any $T>0$ and $  \psi_0 \in H^3(\T,\R)$ such that $\psi_0<0$, there exists $u\in L^2((0,T),\R^{q+2})$, such that $$\psi(T;  \psi_0 ,u)= -\Phi.$$
\end{mtb}

Main Theorem B yields small-time exact controllability to the ground state solution in $H^3$ when $\kappa \geq 0$ and $p\in 2\N$. This property is obtained by using the global $H^s$-approximate controllability result between positive states (Main Theorem A (ii)) together with a $H^3$-local exact controllability result to the ground state solution $\Phi $ in any positive time (see Theorem \ref{teo-loc-nlh} below). The specific choice of the parameters $\kappa$ and $p$ ensures that equation \eqref{0.1} is globally well-posed in $H^3$, which is crucial in our proof of the local exact controllability.

\smallskip

Note that the result of Main Theorem B can be achieved using five potentials. Indeed, one can consider 
$$Q=(1, \cos(x), \sin(x),\mu_1,\mu_2).$$
The potentials $Q_1=1,Q_2=\sin(x),Q_3=\cos(x)$, satisfy Assumption I. Examples of functions $\mu_1$ and $\mu_2$ satisfying Assumption II are (see Example \ref{exemple} below for further details)
$$\mu_1(x)=x^3(2\pi-x)^3, \ \ \ \ \ \ \ \  \ \ \ \ \ \ \mu_2(x)=x^3(x-\pi)^3(x-2\pi)^3.$$

\smallskip

An interesting aspect of Main Theorem B is the validity of an exact controllability result on the torus where the Laplacian exhibits double eigenvalues. Indeed, the first step to prove local exact controllability of \eqref{0.1} is based on the solvability of a suitable moment problem. The method for showing local controllability of bilinear parabolic PDEs, introduced in \cite{acue,acu}, is not directly applicable in our framework. In fact, we face the following two additional difficulties:
\begin{itemize}
    \item the presence of double eigenvalues is, a priori, an obstacle to the solvability of the moment problem at hand. We address the problem by \lq\lq filtering" the spectrum of the Laplacian via the two potentials $\mu_1$ and $\mu_2$. 
We exploit hypotheses \eqref{hp_intro} so that $\mu_1$ acts only on the frequencies associated with the eigenfunctions $\{c_k\}_{k\in\N}$ and $\mu_2$ only on those associated with $\{s_k\}_{k\in\N^*}$. Hence, we decompose the moment problem into two subproblems, each characterized by simple eigenvalues, making both solvable; 
\item in order to adapt the technique proposed in \cite{acu, acue}, the solutions of equation \eqref{0.1} need to be globally defined and unique. These properties may be trivial in the linear case, whereas a more careful analysis should be developed for a nonlinear equation like \eqref{0.1}. We propose a method which requires more regularity on the initial condition and on the control.  
\end{itemize}

\subsection*{Some references}

The classical controllability problem of parabolic-type equations as \eqref{0.1}, with bilinear control, is a delicate matter even in the linear case ($\kappa=0$). The main reason is a structural obstacle, described in \cite{bms}, which makes controllability in $L^2$ impossible. In detail, the results of \cite{bms} imply that the reachable set of \eqref{0.1}, with $\kappa=0$, starting from any $\psi_0\in L^2(\T^d,\R)$, is contained in a countable union of compact subsets of $L^2(\T^d,\R)$, hence it has dense complement. Therefore, this property prevents from obtaining any classical exact controllability result in $L^2$. Hence, we shall explore different notions of controllability, such as the approximate controllability or the exact controllability to trajectories. 

\smallskip

Approximate controllability results via multiplicative controls have been obtained in \cite{Kha} for $1$-D linear parabolic problems in sufficiently large times $T>0$ and between non-negative states. In \cite{Kha}, the control depends on both space and time (see also \cite{Kha1}). Similar results were proved in \cite{ FloCan1} for linear degenerate parabolic problems. Finally, approximate controllability properties of the nonlinear problem were established in \cite{FloCanKah, FloCan, FloNitTro}. 

\smallskip

The techniques leading to approximate controllability of Main Theorem A are inspired by a saturating geometric control approach introduced by Agrachev and Sarychev in \cite{AS-2006, AS-2005} in the case of additive controls. In such works, they proved global approximate controllability of the 2D~Navier-Stokes and 
Euler systems. Their approach  has been extended to study other equations with additive controls (see, for instance, \cite{AS-2008,VN-2019A, Sar-2012,shirikyan-cmp2006}). The saturating control methodology has been implemented for bilinear control problems in the recent work \cite{DucNer} of Nersesyan and the first author. In such paper, small-time approximate controllability between eigenmodes, for a non-linear Schr\"odinger equation, is proved. Main Theorem A is inspired by the approach developed in \cite{DucNer}. Nevertheless, the saturating technique for parabolic equation  as \eqref{0.1} leads to different controllability properties, such as small-time approximate controllability between states that have the same sign. For additional results on small-time controllability of PDEs through bilinear control, achieved using similar methods, we refer also to \cite{small-time-molecule,coron-2023,small-time-momentum,small-time-wave}.

\smallskip

Another controllability property, not ruled out by the negative findings of \cite{bms}, is the exact controllability to trajectories. This property was firstly studied by Alabau-Boussouira, Cannarsa and Urbani in \cite{acue,acu} in an abstract setting for parabolic PDEs by means of a scalar-input bilinear control. 
 An example of application of such results is the local and semi-global controllability of a heat equation as \eqref{0.1} (for $\kappa=0$) on the interval $(0,1)$. 
%
%
The methodology has been later extended in \cite{cdu} to network-type domains to cope with condensation phenomena of the eigenvalues of the diffusion operator. 
The approach of \cite{acue,acu} used for studying local exact controllability to eigensolutions relies on the solvability of a suitable moment problem related to the null controllability of a linearized version of the problem at hand. The resolution of moment problems has been intensively studied in literature over a long period, starting from the classical works \cite{ fattorini, fattorini1}, to the more recent ones \cite{morgan,ammar1,marbach2020, assia1, assia2,cdu,patrick1, patrick2,burgos}.

\smallskip

\subsection*{Scheme of the work}
The paper is organized as follows. In Section \ref{well} we present some preliminary results, namely, local and global well-posedness of the non-linear heat equation \eqref{0.1} and a crucial limit adopted in the proof of Main Theorem A. In Section \ref{SectA} we prove the small-time global approximate controllability and Main Theorem A. Section \ref{local} is devoted to the exact controllability to the ground state solution and contains the proof of Main Theorem B.
In the Appendices \ref{AppA} and \ref{AppB}, we respectively present the proofs of the local and global well-posedness results for equation \eqref{0.1}. Appendix \ref{AppC} contains a technical estimate that we use in the proof of Main Theorem B.

\subsection*{Acknowledgments}
The authors would like to thank the colleagues Vahagn Nersesyan and Tak\'eo Takahashi for the fruitful discussions on the saturation control method adopted in this work. Furthermore, the authors are grateful to the anonymous reviewers for reading the article carefully and giving suggestions to improve it. The third author acknowledges support from INdAM National Group for Mathematical Analysis, Probability and their Applications (GNAMPA) and from the MIUR Excellence Department Project awarded to the Department of Mathematics, University of Rome Tor Vergata, CUP E83C18000100006 (postdoc position 2020-2022). This project has received financial support from the CNRS through the MITI interdisciplinary programs.


\section{Preliminaries}\label{well}
The aim of this section is to present some preliminary results. We start by ensuring existence and uniqueness of solutions of equation \eqref{0.1}. In the second part of the section we prove a limit of conjugated dynamics, which is a key point in the proof of Main Theorem A. 

\subsection{Local and global well-posedness}
We start by stating the following local well-posedness result for the Cauchy problem \eqref{0.1}.

\begin{proposition}\label{well-pos}
    Let $s>d/2$ and $Q\in H^s(\T^d,\R^{q+2})$. For any $\psi_0\in H^s(\T^d,\R)$ and $u\in L^2_{loc}(\mathbb{R}^+,\mathbb{R}^{q+2})$ 
    there exists a maximal time $\mathcal{T}=\mathcal{T}(\psi_0,u)>0$ and a unique mild solution $\psi$ of \eqref{0.1}. Namely, for any $T<\mathcal{T}$, $\psi\in C^0([0,T],H^s(\T^d,\R))$ and is represented by the formula
    \begin{equation*}
        \psi(t;\psi_0,u)=e^{t\Delta}\psi_0+\int_0^t e^{(t-s)\Delta}\left(\langle u(s),Q(x)\rangle\psi(s,x)-\kappa\psi(s,x)^{p+1}\right)ds.
    \end{equation*}
    Moreover, if $\TT<+\infty$, then $\norm{\psi(t)}_{H^s}\to+\infty$ as $t\to \TT^-$. In addition, we have the following properties.
    
    \begin{itemize}
        \item[i.] Assume that $\psi_0,\phi_0\in B_{H^s(\T^d,\R)}(0,R)$ for some $R>0$ and $u,v\in L^2_{loc}(\R^+,\R^{q+2})$. Then, for any $0\leq T\leq \min\{\TT(\psi_0,u),\TT(\phi_0,v)\}$, there exists $C=C(u,v)$ such that
    \begin{equation}\label{psi-phi}
        \sup_{0\leq t\leq T}\norm{\psi(t;\psi_0,u)-\psi(t;\phi_0,v)}_{H^s}\leq C\left(\norm{\psi_0-\phi_0}_{H^s}+\norm{u-v}_{L^2}\right).
    \end{equation}

\item[ii.] Set $ K= \|\psi\|_{C([0,T],H^s)}+\| \psi_0\|_{H^s}+\| u\|_{L^2}$. There exists $\delta=\delta(\TT(\psi_0,u), K)>0$ such that, for any $\hat\psi_0\in H^s(\T^d,\R) $   and $\hat u  \in L^2((0,T), \R^{q+2 })$  satisfying  
 \begin{equation}\label{stability}
\|\hat\psi_0- \psi_0\|_{H^s}+ \|\hat u - u \|_{L^2}   < \delta,
\end{equation} 

problem \eqref{0.1} admits a unique mild solution $\hat\psi\in C\big([0,T],H^s(\T^d,\R)\big)$ with initial condition $\hat \psi_0$ and control $\hat u$.
    \end{itemize}
\end{proposition}

The proof of Proposition \ref{well-pos} follows from a fix point argument and from Sobolev embeddings for $s>d/2$. It can be found in Appendix \ref{AppA}. 

\smallskip
We further present a global well-posedness result for equation \eqref{0.1} in the case $d=1$, $\kappa\geq 0$ and $p\in 2\N$.

\begin{proposition}\label{thm-global-well-pos}
Let $d=1$, $p\in 2\N$, $\psi_0\in H^3(\T,\R)$, $Q\in H^3(\T,\R^{q+2})$, $u\in H^1_{loc}((0,+\infty),\R^{q+2})$ and $\kappa\geq0$. Then, for any $T>0$ there exists a unique mild solution $\psi\in C^0([0,T],H^3(\T,\R))$ of \eqref{0.1}.
\end{proposition}

The proof of Proposition \ref{thm-global-well-pos} follows from some energy estimates and can be found in Appendix \ref{AppB}. 


\subsection{Small-time limit of conjugated dynamics}\label{sandw}
Let us introduce the nonlinear operator
 \begin{equation}\label{eq:saturation-operator}
 	\B(\varphi)(x)=\sum_{j=1}^d\left(\p_{x_j} \varphi(x)\right)^2,\qquad \forall\,\varphi \in C^1(\T^d,\R).
 \end{equation}
Then, the following result holds true.

\begin{proposition}\label{P:1.2}   
 Let $s\in\N^*$, $s>d/2$ and $(Q_1,...,Q_q)\in H^{2s+1}(\T^d,\R^{q})$. Assume that $\psi_0  \in H^s(\T^d,\R)$, $(u_1,...,u_q)\in \R^{q}$, and $\varphi \in H^{2s+1}(\T^d,\R)$ to be non-negative. Then, there exists a constant $\de_0>0$ such that, for any $\de\in (0,\de_0)$, the solution $\psi(t;e^{-\delta^{-1/2}\varphi}\psi_0,\de^{-1} u)$ of \eqref{0.1} with $u=(u_1,...,u_q,0,0)$ is well-defined in $[0,\delta]$. Furthermore, the following limit holds    
\begin{equation*}
 	e^{\delta^{-1/2}\varphi} \psi(\de;e^{-\delta^{-1/2}\varphi}\psi_0,\de^{-1} u)\to  e^{\B(\varphi)+\lag u,Q \rag}\psi_0  \quad\text{in $H^s$, as $\de\to 0^+$}.
\end{equation*}
\end{proposition}
 \begin{proof} {\bf Preliminaries.}
For any $\delta>0$ we set $\phi(t):=  	e^{\delta^{-1/2}\varphi} \psi(t;e^{-\delta^{-1/2}\varphi}\psi_0,\de^{-1} u)$. According to Proposition \ref{well-pos}, $\phi(t)$ exists up to some maximal time $\TT^\delta=\TT(e^{-\delta^{-1/2}\varphi}\psi_0,\de^{-1} u)$ and it has the same regularity of $\psi$. Furthermore, if $\TT^\delta<\ty$
$$
     \| e^{-\delta^{-1/2}\varphi}\phi(t)\|_{H^s}\to +\infty \quad\text{as~$t\to  \TT^{\delta-}$.}
$$
We  introduce the following functions
\begin{align}\label{3.2}
    w(t)  =   e^{\left(\B(\varphi)+\lag u,Q \rag\right)t}\psi_0^\delta,\qquad   v(t)  = \phi(\delta t)-w(t),
\end{align}
where 
 $\psi_0^\delta=e^{\delta^{ 1/4} 
\Delta}\psi_0 \in H^{r}(\T^d,\R)$, with $r=s+2$, is such that 
\begin{align}\label{mollif}\|	\psi_0-\psi_0^\delta\|_{H^s}\to 0,  \qquad \text{ as $\delta\to 0^+$,}\end{align}
and  there exists $C>0$ independent of $\delta>0$ such that
 \begin{align}\label{3.3}
 & \|\psi_0^\delta\|_{H^s}  \le C  ,\qquad\|\psi_0^\delta\|_{H^r} \le C   \delta^{-1/4},  \qquad \text{for } \de\le 1.
 \end{align}
Our aim is to show that $\phi(\delta)\xrightarrow{\delta\rightarrow 0^+}  e^{ \left(\B(\varphi)+\lag u,Q \rag\right)}\psi_0$ in $H^s$. Thanks to \eqref{mollif} it is sufficient to prove that
\begin{align}\label{lim}\|v(t=1)\|_{H^s}\xrightarrow{\delta\rightarrow 0^+}  0  .\end{align}
However, before proving \eqref{lim}, we need to ensure the existence of $\delta_0>0$ small enough such that, for every $0<\delta<\delta_0$, $v (t)$ is well-defined in $[0,1]$, that is,
\begin{align}\label{minim-time}\delta^{-1} \TT^\delta\geq 1.\end{align}

\noindent
 {\bf An intermediate inequality.} In view of \eqref{3.3}, there exists $C>0$ such that, for every $t\in [0,2]$,
\begin{align} 
\|w(t)\|_{H^s}\le C, \quad\,\,\quad\quad \|w(t)\|_{H^r}\le C{\delta^{-1/4}}.\label{3.6}	
\end{align}
We observe that $v$ is a solution of the following equation
\begin{equation}\label{3.7}
\begin{split}
    \p_t v&=\delta\Delta (v+w) - \delta \kappa  (e^{-\delta^{-\frac{1}{2}}\varphi}(v+w))^{p}(v+w)-\delta^{1/2} (v+w) \Delta \varphi\\
	&\quad  - 2 \delta^{\frac12}\nabla  (v+w) \cdot  \nabla \varphi +\B(\varphi)v+\lag u,Q\rag v,
\end{split}
\end{equation}
with initial condition 
     \begin{equation}\label{3.8}
      v(0)=\psi_0-\psi_0^\delta.
      \end{equation}
Let us start by assuming that $\psi_0\in H^{2s+2}(\T^d,\R)$ which implies $\psi(t)\in H^{2s+2}(\T^d,\R)$ and then $v(t)\in H^{2s+2}(\T^d,\R)$ for every $t\in (0,\TT^\delta)$. 
Let $\alpha=(\alpha_1, \ldots, \alpha_d)\in \N^d$ be such that $|\alpha|=|\alpha_1|+\ldots+|\alpha_d|\le s$.
Thanks to the accreativity of the operator $-\Delta$, we get 

\begin{align}
\p_t\| \p^{\alpha}v\|_{L^2}^2&=  -2 \delta \lag (-\Delta \p^{\alpha} v), \p^{\alpha}v\rag_{L^2} +  2\lag  \delta\Delta w+ \delta \kappa  (e^{-\delta^{-\frac{1}{2}}\varphi}(v+w))^{p}(v+w)\nonumber \\
& \quad  - \delta^{\frac{1}{2}} (v+w) \Delta \varphi+ 2 \delta^{\frac12}\nabla  (v+w) \cdot  \nabla \varphi +\B(\varphi)v+\lag u,Q\rag v,\p^{2\alpha }v\rag_{L^2}\nonumber \\ 
& \leq  2\lag  \delta\Delta w+ \delta \kappa  e^{-p\delta^{-\frac{1}{2}}\varphi}(v+w)^{p+1}- \delta^{\frac{1}{2}} (v+w) \Delta \varphi\nonumber \\
& \quad  + 2 \delta^{\frac12}\nabla  (v+w) \cdot  \nabla \varphi +\B(\varphi)v+\lag u,Q\rag v,\p^{2\alpha }v\rag_{L^2}\nonumber \\ 
&\le 2\Big(\delta |\lag \p^{\alpha }\Delta w, \p^{\alpha }v\rag_{L^2}| + \delta  |\kappa|\big|\big\lag  \p^{\alpha}\big( e^{-p\delta^{-\frac{1}{2}}\varphi}(v+w)^{p+1}\big), \p^{\alpha }v\big\rag_{L^2}\big|\nonumber\\ 
&\quad  + \delta^{\frac{1}{2}}|\lag  \p^{\alpha }\big((v+w) \Delta \varphi\big), \p^{\alpha }v\rag_{L^2}|+  \delta^{1/2}|\lag  \p^{\alpha }\big(\nabla  (v+w) \cdot  \nabla \varphi\big), \p^{\alpha }v\rag_{L^2}|\\
&\quad+|\lag \p^{\alpha }\big(\B(\varphi)v\big), \p^{\alpha }v\rag_{L^2}|+|\big \lag \p^{\alpha }\big( \lag u,Q\rag v\big), \p^{\alpha }v\big\rag_{L^2}|\Big).\nonumber
\end{align}
We observe that since all the functions involved in the above estimate are in $H^{2s+2}(\T^d,\R)$, no boundary terms appear when integrating by parts.
Note that there exists a constant $C>0$, independent of $\delta$, such that
$|\lag  \nabla v \cdot  \nabla \varphi, \p^{2\alpha }v\rag_{L^2}|\leq  C\|\varphi\|_{ H^{2s+1}} \|v\|_{H^s}^2$ and
\begin{align*}
    \big|\big\lag  \p^{\alpha}\big( e^{-p\delta^{-\frac{1}{2}}\varphi}(v+w)^{p+1}\big), \p^{\alpha }v\big\rag_{L^2}\big|&\leq C \|e^{-p\delta^{-\frac{1}{2}}\varphi}\|_{L^\infty} \|\varphi\|_{H^{s+1}}^s\Big(p^s\delta^{-\frac{s}{2}} 
    +1\Big) \|v+w\|_{H^s}^{p+1}\|v\|_{H^s}.
\end{align*}
Hence, we deduce the existence of a constant $C>0$, independent of $\delta$, such that
\begin{align*}
    \p_t\| \p^{\alpha}v\|_{L^2}^2&\le C\de \|w\|_{H^r} \|v\|_{H^s}  + C \delta \|e^{-p\delta^{-\frac{1}{2}}\varphi}\|_{L^\infty} (\delta^{-\frac{s}{2}} +1)  \|v+w\|_{H^s}^{p+1} \|v\|_{H^s}+C \delta^{\frac{1}{2}} \|v+w\|_{H^s} \|v\|_{H^s}\nonumber\\ 
    &\quad    + C \delta^{1/2}  \|v\|_{H^s}^2+C \delta^{1/2}  \|w\|_{H^r}\|v\|_{H^s}+C  \|v\|_{H^s}^2+C  \|v\|_{H^s}^2.\nonumber
\end{align*}
Now, we have that $\|e^{-\delta^{-\frac{1}{2}}\varphi}\|_{L^\infty} (\delta^{-\frac{s}{2}} +1) \xrightarrow[]{\delta\rightarrow 0} 0$ thanks to the positivity of $\varphi$. Using \eqref{3.6} and the Young's inequality we deduce that there exists $C>0$, independent of $\delta$, such that
\begin{align*}
    \p_t\| \p^{\alpha}v\|_{L^2}^2&\le C \delta^{3/4} \|v\|_{H^s}  +  C\delta \|v\|_{H^s}^{p+2} + C \delta  \|v\|_{H^s} +C \delta^{\frac{1}{2}} \|v\|_{H^s}^2+C \delta^{\frac{1}{2}}  \|v\|_{H^s}\nonumber\\ 
    &\quad  + C \delta^{1/2}  \|v\|_{H^s}^2+C \delta^{1/4}  \|v\|_{H^s}  
    +C  \|v\|_{H^s}^2+C  \|v\|_{H^s}^2\\
    &\leq C \delta^{1/2}+  C(1+\delta^{1/2})\|v\|_{H^s}^2+ C\delta \|v\|_{H^s}^{p+2} .\nonumber
\end{align*}
The last relation holds for $t\le  \delta^{-1}\TT^\delta$. 
We recall that $\|\cdot\|_{H^s}^2=\sum_{\underset{|\alpha|\le s}{\alpha\in \N^d}}\| \p^{\alpha}\cdot\|_{L^2}^2$ and therefore there exists $C>0$ such that
\begin{align*}
    \p_t\| v\|_{H^s}^2 &\leq C \delta^{1/2}+  C(1+\delta^{1/2})\|v\|_{H^s}^2+ C\delta \|v\|_{H^s}^{p+2} .\nonumber
\end{align*}
By the Gr{\"o}nwall Lemma and recalling \eqref{3.8} we obtain that
\begin{equation}\label{3.10}
    \|v(t)\|_{H^s}^2 \le e^{C(1+\delta^{1/2})t} \left(C \delta^{1/2} t+ \|\psi_0-\psi_0^\delta\|_{H^s}^2+C\delta\int_0^t \|v(y)\|_{H^s}^{p+2}\dd y\right),
\end{equation}
 for $   t\le  \delta^{-1}\TT^\delta$ and for every $\psi_0\in H^{2s+2}(\T^d,\R)$. Finally, we can extend the validity of \eqref{3.10} for every $\psi_0\in H^{s}(\T^d,\R)$ thanks to item \emph{ii.} of Proposition \ref{well-pos} and to the density of 
$H^{2s+2}(\T^d,\R)$ into $H^{s}(\T^d,\R)$ with respect to the $H^{s}$-norm.

\smallskip

\noindent
 {\bf Properties of the maximal time.} It remains to prove \eqref{minim-time}.  We consider $ \delta_0>0$ sufficiently small so that, for $0< \delta<  \delta_0$, we have $\|\psi_0-\psi_0^\delta\|_{H^s}^2< 1/2$ and then
 $$\|v(0)\|_{H^s}^{2}<1/2.$$ 
Denote $\tau^\de:=\sup\left\{t<  \delta^{-1}\TT^\delta: \|v(t)\|_{H^s}<1\right\}$. The above inequality yields $\tau^\de>0$. If $   \tau^\de=+\infty$ then \eqref{minim-time} is obviously satisfied. Thus, let us assume that $$\tau^\de<+\infty.$$ 
To prove \eqref{minim-time} we show that for $ \delta_0>0$ sufficiently small and $0< \delta<  \delta_0$ we have $\tau^\de\geq 1$. Assume by contradiction that, for every $ \delta_0>0$ small, there exists $0< \delta<  \delta_0$ such that $\tau^\de<  1$. Thanks to \eqref{3.10}, we get
\begin{align}\label{contradiction}
    1=\|v(\tau^\de)\|_{H^s}^2< e^{C(1+\delta^{1/2})\tau^\de} \left(C \delta^{1/2} \tau^\de+ \|\psi_0-\psi_0^\delta\|_{H^s}^2+C\delta\int_0^{\tau^\de} \|v(y)\|_{H^s}^{p+2}\dd y\right).
\end{align}
We recall that, by definition, $\|v(t)\|_{H^s}<1$ in $[0,\tau^\de)$. For $\delta_0$ sufficiently small we have
\begin{gather} 
    e^{C(1+\delta^{1/2})\tau^\de}\left(C\delta^{1/2}\tau^\de+\|\psi_0-\psi_0^\delta\|_{H^s}^2\right)<\frac12, \label{3.12}
\end{gather} 
since  $0< \delta<  \delta_0$. Moreover,
\begin{align*}&e^{C(1+\delta^{1/2})\tau^\de} \left(C \delta^{1/2} \tau^\de+ \|\psi_0-\psi_0^\delta\|_{H^s}^2+C\delta\int_0^{\tau^\de} \|v(y)\|_{H^s}^{p+2}\dd y\right)< 1,
\end{align*}
which contradicts \eqref{contradiction}. Hence, we conclude that there exists $\delta_0>0$ sufficiently small such that $\tau^\de> 1$ for every $0<\delta<\delta_0$. Thus, \eqref{minim-time} holds true.

\smallskip

\noindent
 {\bf Conclusion.} Finally, $1\in[0,\tau^\delta)\subset[0,\delta^{-1}\TT^\de)$ and thanks to \eqref{3.10} we have established the validity of \eqref{lim} since
$$
\|v(1)\|_{H^s}^2\le e^{C(1+\delta^{1/2})} \left(C\delta^{1/2}+ \|\psi_0-\psi_0^\delta\|_{H^s}^2+C\delta\right)\to 0 \quad \text{as $\delta\to 0^+$}.
$$  
\end{proof}

A first simple consequence of  Proposition \ref{P:1.2}, concerning the control of \eqref{0.1}, is the following small-time global approximate null controllability.

\begin{corollary}\label{coro_null}
    Let $s\in\N^*$ be such that $s>d/2$ and $1\in{\rm span}\{Q_1,...,Q_q\}$. Assume that $\psi_0\in H^s(\T^d,\R)$. 
    %
%
    For any $\epsilon,T>0$ there exists a constant control $u\in \R^{q+2}$ such that the solution $\psi(t;\psi_0,u)$ of \eqref{0.1} is well-defined in $[0,T]$ and
        $$\|\psi(T;\psi_0,u)\|_{H^s}<\epsilon.$$
\end{corollary}
\begin{proof}
For any $\tilde\epsilon>0$, consider $c>0$ sufficiently large so that 
$$e^{-c}<\frac{\tilde\epsilon}{2\|\psi_0\|_{H^s}} \ \ \ \ \Longrightarrow \ \ \ \ \ \|e^{-c}\psi_0\|_{H^s}<\frac{\tilde\epsilon}{2}.$$
Let $-c=\sum_{j=1}^q u_jQ_j$ for some $(u_1,\dots,u_q)\in \R^q$. Thanks to Proposition \ref{P:1.2} (for $\varphi=0$), there exists $\delta>0$ such that  the constant control $ \hat{u}= (u_1,\dots,u_q,0,0)$ is such that the solution of \eqref{0.1} is well-defined in $[0, \delta ]$ and
$$\|\psi( \delta ;\psi_0,\delta^{-1}\hat{u} )-e^{-c}\psi_0\|_{H^s}<\tilde\epsilon/2.$$
%
Set $v=\hat{u}/\delta$. Thanks to the triangular inequality we get
$$\|\psi( \delta ;\psi_0,v)\|_{H^s}\leq \|\psi( \delta ;\psi_0,v)-e^{-c}\psi_0\|_{H^s}+\|e^{-c}\psi_0\|_{H^s}<\tilde\epsilon.$$
Now, we observe that $0$ is the stationary solution of \eqref{0.1} with control $u=0$. Hence, for $\tilde\epsilon$ sufficiently small, by applying the second item of Proposition \ref{well-pos}, we deduce that $\psi( \cdot ;\psi( \delta ;\psi_0,v), 0)$ and $\psi( \cdot ;0, 0)$ are defined in the same time interval $[\delta,T]$. Furthermore, the first item of Proposition \ref{well-pos} yields the existence of $C>0$, independent of $\psi( \delta ;\psi_0,v)$, such that
$$\|\psi( T ;\psi( \delta ;\psi_0,v), 0)-\psi( T ;0, 0)\|_{H^s}\leq C \|\psi( \delta ;\psi_0,v)\|_{H^s}<C \tilde\epsilon.$$
Thus, for every $\epsilon>0$, it is sufficient to choose $\tilde\epsilon= \frac{\epsilon}{C}$ to conclude the proof.
The control built to achieve the result is defined as
\begin{equation*}
    u(t)=\begin{cases}
        v=\frac{1}{\delta}(u_1,\dots,u_q,0,0) & t\in[0,\delta]\\
        \overline 0 = (0,\dots,0) & t\in[\delta,T].
    \end{cases}
\end{equation*}
\end{proof}

\section{Small-time approximate controllability}\label{SectA}
The aim of this section is to prove the small-time approximate controllability results stated in Main Theorem A.

\subsection{An intermediate controllability result}

Given $Q_1,\dots,Q_q\in C^\infty(\T^d,\R), q\in \N^*$, we introduce the vector space
\begin{equation*}
    \begin{split}\label{H0}
        \mathcal{H}_0={\rm span}_{\R}\{Q_1,\dots,Q_q\}.
    \end{split}
\end{equation*}
We define $\mathcal{H}_j$, for every $j\in\N^*$, as the largest vector space whose elements $\psi$ can be written in the form
\begin{equation}
    \begin{split}\label{Hj}
        \psi=\varphi_0+\sum_{k=1}^n \B(\varphi_k),\qquad \varphi_0,\dots,\varphi_n\in\mathcal{H}_{j-1}, \quad n\in \N,
    \end{split}
\end{equation} 
where the non-linear operator $\B$ is defined in \eqref{eq:saturation-operator}. It is intended that the sum over an empty set of indices, i.e. for $n=0$, is equal to zero.
We also denote 
\begin{equation}\label{H-inf}
\mathcal{H}_\infty=\bigcup_{j=0}^\infty \mathcal{H}_j.
\end{equation}
Let us recall the following result from \cite{DucNer}.

\begin{proposition}{\cite[Proposition 2.6]{DucNer}}\label{P:density}   
Assume that 
$$\{Q_1,\dots,Q_q\}=\{1,\sin\langle k,x\rangle),\cos\langle k,x\rangle\}_{k\in L},$$
for some $L\subset \Z^d$. Then, $\mathcal{H}_\infty$, defined as in \eqref{Hj}-\eqref{H-inf}, is dense in $H^s(\T^d,\R),s\geq 0$, if and only if 
\begin{itemize}
    \item $L$ is a generator,
    \item for any $l,m\in L$, there exists a family $\{n_j\}_{j=1}^\sigma\subset L$ such that $l\not\perp n_1,n_j\not\perp n_{j+1},j=1,\dots,\sigma-1,$ and $ n_\sigma\not\perp m$.
\end{itemize}
\end{proposition}
Observe that, due to Proposition \ref{P:density}, Assumption I guarantees that the choice $L=\mathcal{K}$ (cf. \eqref{ip-trigonotriche}) produces a space $\mathcal{H}_\infty$ that is dense in $H^s(\T^d,\R), s\geq 0$.
We start by ensuring the following property of small-time approximate controllability for \eqref{0.1}. 

\begin{proposition}\label{prop:control}
    Let $s\in\N^*$ be such that $s>d/2$ and $(Q_1,...,Q_q)\in C^\infty(\T^d,\R^q)$ be such that $1\in\mathcal{H}_0$. Assume that $\mathcal{H}_\infty$ is dense in $H^s(\T^d,\R)$. Let $\psi_0\in H^s(\T^d,\R)$ and $\varphi\in H^s(\T^d,\R)$. For any $\epsilon,T>0$, there exist $\tau\in[0,T)$ and $(u_1,...,u_q)\in L^2((0,\tau),\R^q)$  such that the solution $\psi(t;\psi_0,u)$ of \eqref{0.1} with control $u=(u_1,...,u_q,0,0)$ is well-defined in $[0,\tau]$ and
    $$\|\psi(\tau;\psi_0,u)-e^\varphi\psi_0\|_{H^s}<\epsilon.$$
\end{proposition}
\begin{proof}
Let us recall that the concatenation $v*u$ of two scalar control laws $u:[0,T_1]\to \mathbb{R}^{q+2},\ v:[0,T_2]\to \mathbb{R}^{q+2} $ is defined on $[0,T_1+T_2]$ as follows
$$(v*u)(t)=\begin{cases}
u(t), & t\in[0,T_1],\\
v(t-T_1), & t\in(T_1,T_1+T_2].
\end{cases} $$
Such definition extends to controls with values in $\mathbb{R}^{q}$, componentwise. 
We will often use the fact that
$$\psi(T_1+t;\psi_0,v*u)=\psi(t;\psi(T_1,\psi_0,u),v),\quad t>0. $$
Let us start by assuming that the following property holds for any $n\in\mathbb{N}$:
\begin{itemize}
\item[($P_n$)] for any $\psi_0\in H^s(\mathbb{T}^d,\R)$, $\phi\in\mathcal{H}_n$, and any $\varepsilon,T>0$, there exist $\tau\in[0,T)$ and $(u_1,...,u_q):[0,\tau]\to \mathbb{R}^{q}$ piecewise constant such that the solution of \eqref{0.1} associated with the initial condition $\psi_0$ and the control $u=(u_1,...,u_q,0,0)$ satisfies
 $$\left\|\psi(\tau;\psi_0,u)-e^{\phi}\psi_0\right\|_{H^s(\mathbb{T}^d)}< \varepsilon. $$
\end{itemize}

The property $(P_n)$ combined with the density feature implies the statement at once.

\smallskip

We are thus left to prove $(P_n)$. An analogous property appeared in \cite[Theorem 2.2]{DucNer} to study nonlinear Schr\"odinger equations with bilinear controls. We provide the proof for completeness. Let us proceed by induction on the index $n$. \\

\textbf{Inductive basis: $n=0$}\\
If $\phi\in\mathcal{H}_0$, there exists $(u_1,...,u_q)\in\mathbb{R}^{q}$ such that $\phi(x)=\langle u,Q(x)\rangle$ with $u=(u_1,...,u_q,0,0)$. Applying Proposition \ref{P:1.2}, with $\varphi=0$, we deduce that there exists $\tau\in[0,T)$ such that the solution  of \eqref{0.1} associated with the constant control $u^\tau:=u/\tau$ and with the initial condition $\psi_0$ satisfies
$$\left\|\psi(\tau;\psi_0,u^\tau)-e^{\phi}\psi_0\right\|_{H^s(\mathbb{T}^d)}<\varepsilon, $$ 
which proves the desired property.\\

\textbf{Inductive step: $n\Rightarrow n+1$}\\
By assuming that $(P_n)$ holds, we prove $(P_{n+1})$. If $\phi\in\mathcal{H}_{n+1}$, there exist $N\in \mathbb{N}$ and $\phi_0,\dots,\phi_N\in\mathcal{H}_n$ such that
$$\phi=\phi_0+\sum_{j=1}^N \B(\phi_j).$$ 
Consider $\phi_1$ and $c>0$ such that $\tilde \phi_1=\phi_1+c\geq 0$ and note that $\B(\tilde \phi_1)=\B(\phi_1)$. Due to Proposition \ref{P:1.2}, we can find $\delta\in(0,T/3)$ such that
$$\left\|e^{\delta^{-1/2}\tilde \phi_1} \psi(\delta;e^{-\delta^{-1/2}\tilde \phi_1}\psi_0,0)-  e^{\B(\phi_1)}\psi_0\right\|_{H^s}<\varepsilon/2.$$
Since $c\in\mathcal{H}_0$, $\phi_1\in \mathcal{H}_n$, then $\tilde\phi_1\in \mathcal{H}_n$. Thanks to the inductive hypothesis, for any $\varepsilon',T,\delta>0$ there exist $\delta'\in(0,T/3)$ and a piecewise constant control $u^{\delta',\delta}=(u^{\delta',\delta}_1,...,u^{\delta',\delta}_{q},0,0):[0,\delta'] \to \mathbb{R}^{q+2}$ such that 
\begin{equation}\label{eq:first-impulsion}
 \left\|
\psi(\delta';\psi_0,u^{\delta',\delta})-e^{-\delta^{-1/2}\tilde \phi_1}\psi_0\right\|_{H^s}< \varepsilon'. 
\end{equation}
Now, let the dynamics evolve freely in a time interval of length $\delta$, that is, we consider the control $0|_{[0,\delta]}=(0,\dots,0):[0,\delta]\to\mathbb{R}^{q+2}$.
From \eqref{psi-phi} we deduce that there exists $C=C(\delta)$ such that
\begin{multline*}
 \left\|
\psi(\delta'+\delta;\psi_0,0|_{[0,\delta]}*u^{\delta',\delta})-\psi(\delta;e^{-\delta^{-1/2}\tilde \phi_1}\psi_0,0)\right\|_{H^s}
=\left\|\psi(\delta;\psi(\delta';\psi_0,u^{\delta',\delta}),0)-\psi(\delta;e^{-\delta^{-1/2}\tilde \phi_1}\psi_0,0)\right\|_{H^s}< C\varepsilon'. 
\end{multline*}
We use again the inductive hypothesis to deduce that there exist $\delta''\in(0,T/3)$ and a piecewise constant control $u^{\delta'',\delta}=(u^{\delta'',\delta}_1,...,u^{\delta'',\delta}_q,0,0):[0,\delta'']\to \mathbb{R}^{q+2}$ such that
$$\left\|\psi(\delta'';\psi(\delta;e^{-\delta^{-1/2}\tilde \phi_1}\psi_0,0),u^{\delta'',\delta})-e^{\delta^{-1/2}\tilde \phi_1} \psi(\delta;e^{-\delta^{-1/2}\tilde \phi_1}\psi_0,0)\right\|_{H^s}<\varepsilon'. $$
Then, thanks to \eqref{psi-phi}, there exists $C'=C'(\|u^{\delta'',\delta}\|_{L^2},\delta'')$ such that
\begin{equation*}
    \begin{split}
        &\left\|\psi(\delta'+\delta+\delta'';\psi_0,u^{\delta'',\delta}*0|_{[0,\delta]}*u^{\delta',\delta})-e^{\B(\phi_1)}\psi_0\right\|_{H^s}\\
        &\qquad\qquad\qquad\qquad\qquad\leq  \left\| \psi(\delta'';\psi(\delta'+\delta;\psi_0,0|_{[0,\delta]}*u^{\delta',\delta}),u^{\delta'',\delta})-\psi(\delta'';\psi(\delta;e^{-\delta^{-1/2}\tilde \phi_1}\psi_0,0),u^{\delta'',\delta}) \right\|_{H^s}\\
        &\qquad\qquad\qquad\qquad\qquad\qquad\qquad+\left\|\psi(\delta'';\psi(\delta;e^{-\delta^{-1/2}\tilde \phi_1}\psi_0,0),u^{\delta'',\delta})-e^{\delta^{-1/2}\tilde \phi_1} \psi(\delta;e^{-\delta^{-1/2}\tilde \phi_1}\psi_0,0) \right\|_{H^s}\\
        &\qquad\qquad\qquad\qquad\qquad\qquad\qquad+\left\|e^{\delta^{-1/2}\tilde \phi_1} \psi(\delta;e^{-\delta^{-1/2}\tilde \phi_1}\psi_0,0)-e^{\B(\phi_1)}\psi_0 \right\|_{H^s}\\
        &\qquad\qquad\qquad\qquad\qquad\leq C'C\varepsilon'+\varepsilon'+\varepsilon/2.
    \end{split}
\end{equation*}
Choosing $\varepsilon'>0$ small enough such that $C'C\varepsilon'+\varepsilon'<\varepsilon/2$, we have then proved that the piecewise constant control $u^{\delta'',\delta}*0|_{[0,\delta]}*u^{\delta',\delta}$ steers the initial state $\psi_0$ $\varepsilon$-close to the state 
$e^{\B(\phi_1)}\psi_0$ 
in time $\tau:=\delta''+\delta+\delta'<T$.
We can now repeat the same argument for $\phi_2$, reasoning as if we were starting from the initial state $e^{\B(\phi_1)}\psi_0$, and prove that the system can be driven arbitrarily close to the state $e^{\B(\phi_1)+\B(\phi_2)}\psi_0$ in arbitrarily small time, and hence, by iteration, to $e^{\sum_{i=1}^N\B(\phi_i)}\psi_0$. By inductive hypothesis we conclude that there exists a piecewise constant control $u$ leading the state $e^{\sum_{i=1}^N\B(\phi_i)}\psi_0$ arbitrarily close to
$e^{\phi_0+\sum_{i=1}^N\B(\phi_i)}\psi_0$
in arbitrarily small time. This completes the proof of property $(P_n)$.
\end{proof}

 \subsection{Small-time global approximate controllability}\label{main_proofs}

In this section we prove a more general results which in fact has inspired Main Theorem A.

\begin{theorem}\label{mta-gene}

Let $s\in\N^*$, $s>d/2$ and let $(Q_1,...,Q_q)\in C^\infty(\T^d,\R^q)$ be such that $1\in\mathcal{H}_0$ and $\mathcal{H}_\infty$ is dense in $H^s(\T^d,\R)$.

\begin{itemize}
\item[(i)] Let $\psi_0,\psi_1\in H^s(\T^d,\R)$ be such that ${\rm sign}(\psi_0)={\rm sign}(\psi_1)$. 
For any $\epsilon>0$ and $T>0$, there exist $\tau\in (0,T]$ and $(u_1,...,u_q)\in L^2((0,\tau),\R^q)$ for which the solution $\psi(t;\psi_0,u)$ of \eqref{0.1} with control $u=(u_1,...,u_q,0,0)$ is well-defined in $[0,\tau]$ and satisfies
$$
\|\psi(\tau;\psi_0,u)-\psi_1\|_{L^2}<\epsilon.
$$
\item[(ii)] Let $\psi_0,\psi_1\in H^s(\T^d,\R)$ be such that $\psi_0,\psi_1>0$ (or $\psi_0,\psi_1<0$). 
For any $\epsilon>0$ and $T>0$, there exists $(u_1,...,u_q)\in L^2((0,T),\R^q)$ such that the solution $\psi(t;\psi_0,u)$ of \eqref{0.1} with control $u=(u_1,...,u_q,0,0)$ is well-defined in $[0,T]$ and satisfies
$$
\|\psi(T;\psi_0,u)-\psi_1\|_{H^s}<\epsilon.
$$
\end{itemize}
\end{theorem}

\begin{proof}

Let us start by proving \emph{(i)}. Denote by $Z$ the set of zeroes of $\psi_j$, for $ j=0,1$, and its complement in $\T^d$, $Z^c$. 
Consider 
for $\eta>0$ 
the set
$$
Z_\eta:=\{x\in \T^d\mid \dist(x,Z)<\eta\}\supset Z,
$$
and its complement in $\T^d$, $Z^c_{\eta}$. 
For $\eta>0$, we have $Z^c_{\eta} \subset Z^c$ and we define $$\phi_\eta:=\rho_\eta\log(\psi_1/\psi_0),$$ where $\rho^\eta=1$ in $Z^c_{\eta}$ and $0$ elsewhere.
Notice that $\phi_\eta$ is well-defined because $\psi_1/\psi_0>0$ on $Z^c_\eta$, and $\phi_\eta\in L^2(\T^d)$.
We have
$$\|e^{\phi_\eta}\psi_0-\psi_1\|_{L^2(\T^d)}\leq \|e^{\phi_\eta}\psi_0-\psi_1\|_{L^2(Z^c_\eta)} + \|\psi_0-\psi_1\|_{L^2(Z_\eta)}=\|\psi_0-\psi_1\|_{L^2(Z_\eta\setminus Z)}.$$
Hence, fixed any $\varepsilon,T>0$, we can choose $\eta>0$ small enough such that 
$$
\|e^{\phi_\eta}\psi_0-\psi_1\|_{L^2(\T^d)}<\varepsilon/3.
$$
Now, by the density of $H^1(\T^d)$ into $L^2(\T^d)$, there exists $\widetilde \phi_\eta\in H^1(\T^d)$ such that
$$
\|e^{\widetilde \phi_\eta}\psi_0-\psi_1\|_{L^2(\T^d)}\leq \|e^{\widetilde\phi_\eta}\psi_0-e^{\phi_\eta}\psi_0\|_{L^2(\T^d)}+\|e^{ \phi_\eta}\psi_0-\psi_1\|_{L^2(\T^d)}  <\frac{2\varepsilon}{3}.
$$
Then, we apply Proposition \ref{prop:control} with $\varphi=\widetilde\phi_\eta$ and we deduce that there exist a time $\tau\in[0,T)$ and a control $u=(u_1,...,u_q,0,0)\in L^2((0,\tau),\R^{q+2})$ such that the solution $\psi(t;\psi_0,u)$ of \eqref{0.1} is well-defined in $[0,\tau]$ and 
$$
\|\psi(\tau;\psi_0,u)-e^\varphi\psi_0\|_{L^2}<\frac{\varepsilon}{3}.
$$
By the triangular inequality we conclude that
$$
\|\psi(\tau;\psi_0,u)-\psi_1\|_{L^2}\leq \|\psi(\tau;\psi_0,u)-e^\varphi\psi_0\|_{L^2}+\|e^\varphi\psi_0-\psi_1\|_{L^2}<\varepsilon,
$$
and the first item of the Proposition is then proved.

\smallskip

The proof of \emph{(ii)} follows from the same strategy used to prove \emph{(i)}. However, we now consider the $H^s$-norm instead of the $L^2$-norm. 
In this case, we substitute $\phi_\eta$ with two functions $\phi_1 =\log(\sign(\psi_0)/\psi_0)$ and $\phi_2=\log(\sign(\psi_1)\psi_1)$, which are well-defined everywhere in $\T^d$. Note that, according to our assumptions, $\phi_1,\phi_2\in H^s(\T^d,\R)$.
Proposition \ref{prop:control} with $\varphi=\phi_1$ yields the existence of a control $u^1=(u^1_1,...,u^1_q,0,0):[0,\tau_1)\rightarrow\R^{q+2}$ that steers $\psi_0$ close to the constant state $1$ in time $\tau_1\leq T/2$. We apply again Proposition \ref{prop:control} with $\varphi=\phi_2$. Thus, we can find a control $u^2=(u^2_1,...,u^2_q,0,0):[0,\tau_2)\rightarrow\R^{q+2}$ such that the solution of \eqref{0.1} starting from a state close to $1$ at time $\tau_2\leq T/2$ reaches a neighbourhood of the final target $\psi_1$.
Hence, the approximate controllability in $H^s$ can be achieved in time $T$ by exploiting that $1$ is a stationary solution of \eqref{0.1} associated with the control $u^{stat}=(u^{stat}_1,...,u^{stat}_q,0,0):[0,T-\tau_1 -\tau_2]\rightarrow\R^{q+2}$, thanks to the assumption that $1\in\mathcal{H}_0$. Therefore, the control will be defined as
$$(u^2* u^{stat}*u^1)(t) = {\mathbf 1}_{[0,\tau_1)} u^1(t) + {\mathbf 1}_{[\tau_1,T-\tau_2]} u^{stat}(t-\tau_1) + {\mathbf 1}_{(T-\tau_2,T]} u^{2}(t-T+\tau_2).$$
\end{proof}

\begin{remark} The reason why the approximate controllability of part (i) of Theorem \ref{mta-gene} (and consequently, of Main Theorem A) is only stated with respect to the $L^2$-norm, while part (ii) is stated with respect to the finer $H^s$-norm, is due to our approximation technique. More precisely, the term $\|e^{\phi_\eta}\psi_0-\psi_1\|_{H^s(Z^c_\eta)}$ (which appears in the proof above) cannot be small as $\eta\to 0$, when $s>0$ and $Z\neq \emptyset$.
\end{remark}

\begin{proof}[Proof of Main Theorem A]
Main Theorem A follows from Theorem \ref{mta-gene}. Indeed, thanks to Proposition \ref{P:density}, the space $\mathcal{H}_\infty$ is dense in $H^s(\T^d,\R)$. Furthermore, the potentials $Q_j$, $j=1,\dots,q$, belong to $H^r(\T^d,\R)$ for every $r>0$. Therefore, the proof is completed.  
\end{proof}


\section{Exact controllability to the ground state solution}\label{local}

We now move to the proof of Main Theorem B. Henceforth, we suppose $d=s=1$. Let the solution of \eqref{0.1} exist for any time $T>0$. For instance, it is enough to require $p\in 2\N$, $\kappa\geq0$, and $\psi_0\in H^3(\T)$, as stated in Proposition \ref{thm-global-well-pos}. Main Theorem B will be a direct consequence of the following more general result.

\begin{theorem}\label{mtb-gene}
Let $\kappa\geq0$ and $p\in 2\N$. Suppose that Assumption II is satisfied and $\mathcal{H}_\infty$ is dense in $H^3(\T,\R)$. Then, \eqref{0.1} is exactly controllable to the ground state solution $\Phi=c_0$, defined in \eqref{eigenf}, in any positive time from any positive state. In other words, for any $T>0$ and
$$ \psi_0 \in \{\psi\in H^3(\T,\R) :\ sign(\psi)>0\},$$
there exists $u\in L^2((0,T),\R^{q+2})$ such that $$\psi(T;  \psi_0 ,u)= \Phi.$$
Analogously, for any $T>0$ and $$ \psi_0 \in \{\psi\in H^3(\T,\R) :\ sign(\psi)<0\},$$ there exists $u\in L^2((0,T),\R^{q+2})$ such that $$\psi(T;  \psi_0 ,u)= -\Phi.$$
\end{theorem}

\begin{example}\label{exemple}
Examples of suitable functions $\mu_1$ and $\mu_2$ satisfying hypotheses \eqref{hp_intro} in Assumption II are
$$\mu_1=x^3(2\pi-x)^3, \qquad \mu_2=x^3(x-\pi)^3(x-2\pi)^3.$$
Indeed, both functions belong to $H^3(\T,\R)$, $\mu_1$ is symmetric with respect to $x=\pi$ and $\mu_2$ is antisymmetric. Thus, $\lag\mu_1,c_0\rag_{L^2}\neq 0 $ and $\lag\mu_2,c_0\rag_{L^2}=0 $. Moreover,\begin{equation*}\begin{split}
 \lag \mu_1 , s_{k}\rag_{L^2}=0,\qquad \lag \mu_2 , c_{k}\rag_{L^2}=0.\\
\end{split}\end{equation*}
By further computations, the remaining hypotheses in \eqref{hp_intro} are satisfied since
 \begin{equation*}\begin{split}
  \lag \mu_1, c_{k}\rag_{L^2} =\frac{96\pi(k^2\pi^2-15) }{ k^6},\qquad   \lag \mu_2, s_{k}\rag_{L^2} =\frac{-864 \pi (840-105 k^2 \pi^2+2 k^4 \pi^4)}{k^9}.\\
\end{split}\end{equation*}
\end{example}

Before proving Theorem \ref{mtb-gene}, and consecutively Main Theorem B, we need to ensure the following local exact controllability result to the ground state solution $\Phi$.

\begin{theorem}\label{teo-loc-nlh}
Let $\kappa\geq0$ and $p\in 2\N$. Suppose that Assumption II is satisfied. Then, \eqref{0.1} is locally exactly controllable to the ground state solution $\Phi=c_0$, defined in \eqref{eigenf}, in any positive time. In other words, for any $T>0$ there exists $R_T>0$ such that, for any 
$$\psi_0\in \{\psi\in H^3(\T,\R): \|\psi-\Phi\|_{H^1}<R_T\},$$
there exists $(u_1,u_2)\in H^1((0,T),\R^2)$ such that $\psi(T; \psi_0,u)=\Phi$, where $u = (\frac{\kappa}{\Phi^p},0,...,0,u_1,u_2)$.  Furthermore,
\begin{equation}\label{bound-control}
\norm{u}_{H^1(0,T)}\leq \frac{e^{-\pi^2\Gamma_0/T}}{e^{2\pi^2\Gamma_0/(3T)}-1},
\end{equation}
where
\begin{equation}\label{Gamma_0}
    \Gamma_0:=2\nu+(\max\{\ln\gamma_1,0\}+\max\{\ln C^2_Q,0\}+\gamma_2+\ln 8)/2,
\end{equation}
with
\begin{equation*}\label{gamma_1-gamma_2}
    \gamma_1:=2\kappa(p+1)^2\sum_{j=2}^{p+1}\binom{p+1}{j}\Phi^{p+1-j},\qquad\gamma_2:=2\kappa(p+1)\Phi^p+\sum_{j=2}^{p+1}\binom{p+1}{j}\Phi^{p+1-j}+1,
\end{equation*}
and 
\begin{equation}\label{R_T}
    R_T=e^{-6\Gamma_0/T_1},
\end{equation}
with
\begin{equation*}\label{T_1}
    T_1:=\min\Big\{\frac{6}{\pi^2}T,1,T_0\Big\}.
\end{equation*}
The constant $T_0$ is defined in \eqref{estim-control-time} and $C_Q$ in \eqref{c_Q}.
\end{theorem}


\subsection{Control of the linearized system}

The result of Theorem \ref{teo-loc-nlh} follows from the null controllability of an associated linear system and the iteration of a control procedure on a clever choice of time steps, as proposed in \cite{acue}.  

We first observe that the ground state $\Phi=1/\sqrt{2\pi}$ is solution  of \eqref{0.1} in any time interval $[0,T]$ for $u=(u_\kappa,0,0,\dots,0)$, with $u_\kappa=\frac{\kappa}{(2\pi)^{p/2}}$ and $\psi_0=\Phi$. Indeed, we recall that $Q_1=1$, thanks to Assumption II. 

Let $s_1,s_0>0$. Consider the following linear system
\begin{equation}\label{lin-intro}
\begin{cases}
\partial_t\xi(t)-\partial_x^2 \xi (t) +\kappa p\xi (t) = \lag v(t), Q\rag \Phi, &t\in(s_0,s_1),\\
\xi(s_0)=\xi_0,
\end{cases}
\end{equation}
with $Q=(Q_1,...,Q_q,\mu_1,\mu_2)\in H^3(\T,\R^{q+2})$ and $q\in\N^*$. We denote by $\xi(\cdot;s_0,\xi_0,v)$ the solution of \eqref{lin-intro} with initial condition $\xi_0$ at time $s_0$ and control $v$.


\begin{definition}
The pair $(-\partial^2_x,Q)$ is said to be $1$-\emph{null controllable} in time $T>s_0$ if there exists a constant $N(T)>0$ such that for any $\xi_0\in L^2(\T,\R)$, there exists a control $v\in L^2((s_0,T),\R^{q+2})$ such that $\xi(T;s_0,\xi_0,v)=0$ and moreover $\norm{v}_{L^2(0,T)}\leq N(T)\norm{\xi_0}_{L^2}$. The best constant, that is,
\begin{equation*}\label{control-cost}
N(T):=\sup_{\norm{\xi_0}_{L^2}=1}\inf \left\{\norm{v}_{L^2}\,:\,\xi(T;s_0,\xi_0,v)=0\right\}
\end{equation*}
is called the \emph{control cost}. If $v\in H^1_0((s_1,T),\R^{q+2}),$ the pair $(-\partial^2_x,Q)$ is called \emph{smoothly} $1$-null controllable and
\begin{equation*}\label{control-cost_1}
N(T):=\sup_{\norm{\xi_0}_{L^2}=1}\inf \left\{\norm{v}_{H^1}\,:\,\xi(T;s_0,\xi_0,v)=0\right\}.
\end{equation*}
\end{definition}


\begin{remark}\label{reg-sol-xi}
    Let us observe that, fixed $v\in L^2((s_0,s_1),\R^{q+2})$ and $Q\in H^1(\T,\R^{q+2})$, for any $\xi_0\in L^2(\T,\R)$, there exists a unique mild solution $$\xi\in C([s_0,s_1],L^2(\T))$$ of \eqref{lin-intro} (see \cite[Proposition 2.1]{acue} adapted to the current case). Furthermore, given a smoother initial condition of \eqref{lin-intro}, for instance $\xi_0\in H^1(\T)$, the solution $\xi$ belongs to the space $$C([s_0,s_1],H^1(\T))\cap H^1((s_0,s_1),L^2(\T))\cap L^2((s_0,s_2),H^2(\T)).$$ Such regularity is called \emph{maximal regularity}, and it is due to the analiticity of the semigroup generated by the operator $\partial^2_x$ (see, for instance, \cite[Proposition 4]{acue} and \cite[Corollary 3]{acu-fp} that can be adapted to \eqref{lin-intro}). 
\end{remark}

The following result shows that the last two components of the control are used to drive the linear system \eqref{lin-intro} to rest. From now on, unless stated otherwise, constants C without indices may vary from line to line.

\begin{proposition}\label{prop-null}

Suppose that Assumption II is satisfied. Then, $(-\partial^2_x,Q)$ is smoothly null controllable in any time $T>0$. Namely, for any $\xi_0\in L^2(\T,\R)$, there exists $(v_1,v_2)\in H^1_0((0,T),\R^2)$ such that the solution of \eqref{lin-intro} with $(s_0,s_1)=(0,T)$ and $v=(0,...,0, v_1,v_2)$ satisfies $$\xi(T;\xi_0, v)=0.$$
Furthermore, there exist $\nu,T_0>0$ such that
\begin{equation}\label{estim-control-time}
N(\tau)\leq e^{\nu/\tau},\quad\forall\,0<\tau\leq T_0.
\end{equation}
\end{proposition}
\begin{proof}
For all $T>0$, we want to prove the existence of $N(T)>0$ such that for any initial condition $\xi_0\in L^2(\T,\R)$ there exists a control $ v=(0,...,0,v_1,v_2)$ so that $\xi(T;0,\xi_0, v)=0$. To this purpose, we first note that the operator $-\partial_x^2 + p\kappa$ exhibits double eigenvalues $\tilde\lambda_j=\lambda_j+p\kappa$ with $j\in\N$ (see definition \eqref{eigenv} of $\lambda_j$) -- with the exception of the first one -- associated to the eigenfunctions $\{c_0,c_j,s_j\}_{j\in\N^*}$, defined in \eqref{eigenf}.

Consider the reduced problem \eqref{lin-intro} with $\hat{Q}=(\mu_1,\mu_2)$ and $(s_0,s_1)=(0,T)$. We look for a control of the form $\hat v=(v_1,v_2)$. We decompose the solution with respect to the Hilbert basis $\{c_0,c_k,s_k\}_{k\in\N^*}$. Recalling that $\lag \mu_1,s_k \rag_{L^2} =\lag \mu_2,c_k \rag_{L^2} = 0$, the smooth null controllability property is equivalent to finding $\hat v\in H^1_0((0,T);\R^2)$ such that
\begin{multline*}
    0=\xi(T;0,\xi_0,\hat v)=\sum_{k\in\N}e^{-\tilde \lambda_k T}\lag \xi_0,c_k\rag_{L^2} c_k+\sum_{k\in\N^*}e^{-\tilde \lambda_k T}\lag \xi_0,s_k\rag_{L^2} s_k \\
    -\int_0^T  v_1(s)\sum_{k\in\N}e^{-\tilde \lambda_k(T-s)}\lag \mu_1\Phi,c_k\rag_{L^2} c_k\,ds-\int_0^T v_2(s)\sum_{k\in\N^*}e^{-\tilde \lambda_k(T-s)}\lag \mu_2\Phi,s_k\rag_{L^2} s_k\,ds.
\end{multline*}
Note that $\Phi=c_0$ is a constant. The above equation is satisfied when the following infinite number of identities hold true for a control function $\hat v=(v_1,v_2)$
\begin{equation*}\begin{split}
    \lag \xi_0,c_k\rag_{L^2}=c_0\int_0^T e^{\tilde \lambda_k s}v_1(s)\lag \mu_1 ,c_k\rag_{L^2}\,ds,\quad\forall\, k\in\N,\\
    \lag \xi_0,s_k\rag_{L^2}=c_0\int_0^T e^{\tilde \lambda_k s}v_2(s)\lag \mu_2 ,s_k\rag_{L^2}\,ds,\quad\forall\, k\in\N^*,
\end{split}\end{equation*}
that can be rewritten in compact form as
\begin{align}
    \int_0^T e^{ \lambda_k s}\tilde v_1(s)\,ds=d_k^1,\quad\forall\,k\in\N,\label{mom-prob1}\\
        \int_0^T e^{ \lambda_k s}\tilde v_2(s)\,ds=d_k^2,\quad\forall\,k\in\N^*,\label{mom-prob2}
\end{align}
where  
$$d_k^1:=\frac{\lag \xi_0,c_k\rag_{L^2}}{c_0\lag \mu_1,c_k\rag_{L^2}}\quad\text{with}\quad \,k\in\N,\quad\quad\quad d_k^2:=\frac{\lag \xi_0,s_k\rag_{L^2}}{c_0\lag \mu_2,s_k\rag_{L^2}}\quad\text{with}\quad \,k\in\N^*, $$
and
\begin{equation*}
    \tilde v_1(s)=e^{p\kappa s}v_1(s),\qquad \tilde v_2(s)=e^{p\kappa s}v_2(s).
\end{equation*}
Note that $d_k^1$, $k\in\N$, and $d_k^2$, $k\in\N^*$, are well-defined thanks to Assumption II. We treat \eqref{mom-prob1} and \eqref{mom-prob2} as two separate moment problems. Let us start by solving \eqref{mom-prob1}. We seek for a function $\tilde v_1$ such that
\begin{equation*}
    \tilde v_1'(t)=r_1\left(\frac{T}{2}-t\right)e^{-(\frac{T}{2}-t)},\qquad r_1\in L^2\left(-\frac{T}{2},\frac{T}{2}\right).
\end{equation*}
Therefore, it should hold
\begin{equation*}
    \tilde v_1(t)-\tilde v_1(0)=\int_0^t \tilde v_1'(s)ds=\int_0^t r_1\left(\frac{T}{2}-s\right)e^{-(\frac{T}{2}-s)}ds.
\end{equation*}
We require that $\tilde v_1(0)=\tilde v_1(T)=0$, which implies that $r_1$ must satisfy
\begin{equation}\label{eq-int-r1e}
    \int_{-\frac{T}{2}}^{\frac{T}{2}}r_1(s)e^{-s}ds=0.
\end{equation}
Integrating by parts \eqref{mom-prob1} and taking into account that $\tilde v_1(0)=\tilde v_1(T)=0$, we get for any $k\in\N^*$
\begin{equation*}
    \begin{split}
        d^1_k&=\int_0^Te^{\lambda_k s}\tilde v_1(s)ds=\frac{1}{\lambda}_k\left(\left.e^{\lambda_k t}\tilde v_1(t)\right|^T_0-\int_0^T e^{-(\frac{T}{2}-t)+\lambda_k t} r_1\left(\frac{T}{2}-t\right)dt\right)\\
        &=\frac{1}{\lambda_k}\int_{-\frac{T}{2}}^{\frac{T}{2}}e^{-s+\lambda_k(\frac{T}{2}-s)}r_1(s)ds.
    \end{split}
\end{equation*}
The above identities can be rewritten as follows
\begin{equation*}
    \int_{-\frac{T}{2}}^{\frac{T}{2}}e^{-(1+\lambda_k)s}r_1(s)ds=\lambda_k d^1_k e^{-\frac{T}{2}\lambda_k},\qquad k\in\N^*.
\end{equation*}
For $k=0$ we have that
\begin{equation*}
    \begin{split}
        d^1_0&=\int_0^T\tilde{v}_1(s)ds=\left. s \tilde v_1(s)\right|^T_0-\int_0^Ts r_1\left(\frac{T}{2}-s\right)e^{-(\frac{T}{2}-s)}ds=\int_{-\frac{T}{2}}^{\frac{T}{2}}\left(\frac{T}{2}-s\right)r_1(s)e^{-s}ds\\
        &=-\int_{-\frac{T}{2}}^{\frac{T}{2}} s e^{-s} r_1(s)ds,
    \end{split}
\end{equation*}
where we have used \eqref{eq-int-r1e}. Define the family
\begin{equation*}
    \omega_k:=1+\lambda_k,\qquad k\in\N^*,
\end{equation*}
and the sequence $\tilde d^1_{kj}$
\begin{equation*}
    \tilde d^1_{01}=-d_0,\qquad \tilde d^1_{k1}=0,\qquad k\in\N^*,
\end{equation*}
\begin{equation*}
\tilde{d}^1_{00}=0,\qquad\tilde{d}^1_{k0}=\lambda_kd^1_ke^{-\frac{T}{2}\lambda_k},\qquad k\in\N^*.
\end{equation*}
Thus, we look for $r_1\in L^2\left(-\frac{T}{2},\frac{T}{2}\right)$ such that
\begin{equation}\label{mom-prob-r1}
    \int_{-\frac{T}{2}}^{\frac{T}{2}}s^je^{-\omega_k s}r_1(s)ds=\tilde{d}^1_{kj},\qquad j=0,1,\quad k\in\N.
\end{equation}
As proved in \cite[Theorem 1.5]{assia1}, there exists $T_0>0$ such that for every $0<T<T_0$ there exists a family
\begin{equation*}
    \{\sigma_{k,j}\}_{k,j\in\N}\subset L^2\left(-\frac{T}{2},\frac{T}{2}\right),
\end{equation*}
which is biorthogonal to
\begin{equation*}
    e_{kj}(s)=s^je^{-\omega_k s},\qquad k,j\in\N,
\end{equation*}
and moreover
\begin{equation*}
    \norm{\sigma_{k,j}}_{L^2(-T/2,T/2)}\leq Ce^{C\sqrt{\omega_k}+\frac{C}{T}},\qquad j=0,1,\quad k\in\N^*.
\end{equation*}
Therefore, by defining
\begin{equation*}
    r_1(s):=\tilde{d}^1_{01}\sigma_{01}+\sum_{k=1}^\infty \tilde{d}^1_{k0}\sigma_{k0}(s),
\end{equation*}
we deduce that $r_1$ solves the moment problem \eqref{mom-prob-r1}. Let us finally show that $r_1\in L^2\left(-\frac{T}{2},\frac{T}{2}\right)$:
\begin{equation*}
    \begin{split}
        \norm{r_1}_{L^2(-T/2,T/2)}&\leq |\tilde d^1_{01}|\norm{\sigma_{01}}_{L^2(-T/2,T/2)}+\sum_{k=1}^\infty|\tilde d^1_{k0}|\norm{\sigma_{k0}}_{L^2(-T/2,T/2)}\\
        &\leq Ce^{C/T}|\tilde d^1_{01}|+C\sum_{k=1}^\infty\lambda_k|d^1_k|e^{-\frac{T}{2}\lambda_k+C\sqrt{1+\lambda_k}+\frac{C}{T}}\\
        &\leq Ce^{C/T}\left(\frac{|\lag \xi_0,c_0\rag_{L^2}|}{c_0|\lag \mu_1,c_0\rag_{L^2}|}+\left(\sum_{k=1}^\infty \frac{\lambda_k^2 e^{-T\lambda_k+C\sqrt{\lambda_k+1}} }{c_0^2|\lag \mu_1,c_k\rag_{L^2}|^2}\right)^{1/2}\norm{\xi_0}_{L^2}\right)\\
        & \leq Ce^{C/T}\left(1+\left(\sum_{k=1}^\infty \lambda_k^{2(q_1+1)} e^{-T\lambda_k+C\sqrt{\lambda_k+1}}\right)^{1/2}\right)\norm{\xi_0}_{L^2},
    \end{split} 
\end{equation*}
where $q_1$ is the parameter introduced in Assumption II. Let us analyse the behaviour with respect to $T$ of the following series:
\begin{equation}
S(T):=\sum_{k=1}^\infty \lambda_k^{2(q_1+1)} e^{-T\lambda_k+C\sqrt{\lambda_k+1}}=\sum_{k=1}^\infty\left(\lambda^{2(q_1+1)}_ke^{-\frac{T}{2}\lambda_k}\right)e^{-\frac{T}{2}\lambda_k+C\sqrt{\lambda_k+1}}.
\end{equation}
For any $\lambda\geq 0$, we introduce the function $f(\lambda):=e^{-\frac{T}{2}\lambda+C\sqrt{\lambda+1}}$. Its derivative is given by
\begin{equation}
f'(\lambda)=e^{-\frac{T}{2}\lambda+C\sqrt{\lambda}}\left(-\frac{T}{2}+\frac{C}{2\sqrt{\lambda+1}}\right),
\end{equation}
and its maximum is attended at $\lambda=\left(\frac{C}{T}\right)^2-1$. Hence, for every $0<T\leq1$ we have
\begin{equation}\label{bound-S(T)}
S(T)\leq e^{\frac{T}{2}+\frac{C}{T}}\sum_{k= 1}^\infty \left(\lambda_k^{2(q_1+1)} e^{-\frac{T}{2}\lambda_k}\right)\leq e^{\frac{C}{T}}\sum_{k=1}^\infty \left(\lambda_k^{2(q_1+1)} e^{-\frac{T}{2}\lambda_k}\right).
\end{equation}
Now, for any $\lambda\geq0$, we consider the function $g(\lambda):=\lambda^{2(q_1+1)}e^{-\frac{T}{2}\lambda}$. From its derivative $$g'(\lambda)=\lambda^{2q_1+1}e^{-\frac{T}{2}\lambda}\left(2(q_1+1)-\frac{T}{2}\lambda\right),$$ we deduce that
\begin{equation}
g(\lambda)\text{ is }\left\{\begin{array}{ll}
\text{increasing }&\text{if }0\leq \lambda<\frac{4(q_1+1)}{T},\\\\
\text{decreasing }&\text{if }\lambda\geq \frac{4(q_1+1)}{T},
\end{array}\right.
\end{equation}
and $g$ has a maximum for $\lambda=\frac{4(q_1+1)}{T}$. We define the index
\begin{equation}
k_1:=k_1(T)=\sup \left\{k\in\N^* \,:\, \lambda_k\leq\frac{4(q_1+1)}{T}\right\}.
\end{equation}
We can rewrite the sum in \eqref{bound-S(T)} as
\begin{equation}
\sum_{k=1}^\infty \lambda_k^{2(q_1+1)}e^{-\frac{T}{2}\lambda_k}=\sum_{1\leq k\leq k_1-1}\lambda_k^{2(q_1+1)}e^{-\frac{T}{2}\lambda_k}+\sum_{k_1\leq k	\leq k_1+1}\lambda_k^{2(q_1+1)}e^{-\frac{T}{2}\lambda_k}+\sum_{k\geq k_1+2}\lambda_k^{2(q_1+1)}e^{-\frac{T}{2}\lambda_k}.
\end{equation}
For any $1\leq k\leq k_1-1$, we have
\begin{equation}
\int_{\lambda_k}^{\lambda_{k+1}}\lambda^{2(q_1+1)}e^{-\frac{T}{2}\lambda}d\lambda\geq (\lambda_{k+1}-\lambda_k)\lambda_k^{2(q_1+1)}e^{-\frac{T}{2}\lambda_k}\geq \lambda_k^{2(q_1+1)}e^{-\frac{T}{2}\lambda_k},
\end{equation}
and for any $k\geq k_1+2$, it holds that
\begin{equation}
\int_{\lambda_{k-1}}^{\lambda_{k}}\lambda^{2(q_1+1)}e^{-\frac{T}{2}\lambda}d\lambda\geq (\lambda_k-\lambda_{k-1})\lambda_k^{2(q_1+1)}e^{-\frac{T}{2}\lambda_k}\geq  \lambda_k^{2(q_1+1)}e^{-\frac{T}{2}\lambda_k}.
\end{equation}
Therefore, we obtain that 
\begin{equation}\label{int-plus-sum}
\sum_{k=1}^\infty \lambda_k^{2(q_1+1)}e^{-\frac{T}{2}\lambda_k}\leq 2\int_0^\infty \lambda^{2(q_1+1)}e^{-\frac{T}{2}\lambda}d\lambda+\sum_{k_1\leq k	\leq k_1+1}\lambda_k^{2(q_1+1)}e^{-\frac{T}{2}\lambda_k}.
\end{equation}
Recalling that $g$ has a maximum at $\lambda=\frac{4(q_1+1)}{T}$, we have that
\begin{equation}\label{k1,k1+1}
\lambda_k^{2(q_1+1)}e^{-\frac{T}{2}\lambda_k}\leq \left(\frac{4(q_1+1)}{T}\right)^{2(q_1+1)}e^{-2(q_1+1)},\quad\text{for }k=k_1,k_1+1.
\end{equation}
Moreover, we can rewrite the integral term in \eqref{int-plus-sum} as
\begin{equation}\label{euler-integral}
\begin{split}
\int_0^\infty \lambda^{2(q_1+1)}e^{-\frac{T}{2}\lambda}d\lambda&=\frac{2}{T}\int_0^\infty \left(\frac{2s}{T}\right)^{2(q_1+1)}e^{-s}ds\\
&= \left(\frac{2}{T}\right)^{2q_1+3}\int_0^\infty s^{2(q_1+1)}e^{-s}ds=\Gamma(2q_1+3) \left(\frac{2}{T}\right)^{2q_1+3},
\end{split}
\end{equation}
where by $\Gamma(\cdot)$ we indicate the Euler integral of the second kind.
Therefore, thanks to \eqref{k1,k1+1} and \eqref{euler-integral} we conclude that there exist two positive constants $C^1_{q_1},C^2_{q_1}$ such that
\begin{equation}
\sum_{k=1}^\infty \lambda_k^{2(q_1+1)}e^{-\frac{T}{2}\lambda_k}\leq \frac{C^1_{q_1}}{T^{2q_1+2}}+\frac{C^2_{q_1}}{T^{2q_1+3}}.
\end{equation}
From the above estimate we can prove that there exist positive constants $C$ for which
\begin{equation}
\norm{r_1}_{L^2(-T/2,T/2)}\leq Ce^{C/T}\norm{\xi_0}_{L^2}, \quad\forall\,0<T<\min\{T_0,1\},
\end{equation}
and thus 
\begin{equation}
\norm{(\tilde v_1)'}_{L^2(0,T)}\leq e^{T/2}\norm{r_1}_{L^2(-T/2,T/2)}\leq Ce^{C/T}\norm{\xi}_0,\quad\forall\,0<T<\min\{T_0,1\}.
\end{equation}
From the Poincar{\'e} inequality, there exists a constant $C>0$ such that for any $\tilde v\in H^1_0(0,T)$ it holds that $\norm{\tilde v_1}_{H^1(0,T)}\leq C\norm{(\tilde v_1)'}_{L^2(0,T)}$, and we deduce that
\begin{equation}
\norm{\tilde v_1}_{H^1(0,T)}\leq C\norm{(\tilde v_1)'}_{L^2(0,T)}\leq Ce^{C/T}\norm{\xi_0}_{L^2},\quad\forall\,0<T<\min\{T_0,1\}.
\end{equation}
Finally, recalling that $v_1(t)=e^{-p\kappa t}\tilde v_1(t)$, we conclude that
\begin{equation*}
\begin{split}
\norm{v_1}^2_{H^1(0,T)}&=\int_0^Te^{-2p\kappa t}|\tilde v_1(t)|^2dt+\int_0^T\left(\left(e^{-2pk t}|\tilde v_1(t)|\right)'\right)^2dt\\
&\leq C\norm{\tilde v_1}^2_{H^1(0,T)}\leq Ce^{C/T}\norm{\xi_0}_{L^2}^2,
\end{split}
\end{equation*}
for all $0<T\leq\min\{T_0,1\}$.

The same computations are also valid for the second component of the control $v_2$ and the proof is then completed. 
\end{proof}

\begin{remark}\label{nuovo_revisione}
The choice of Assumption II in Proposition \ref{prop-null} is suitable to deduce the controllability result stated in Theorem \ref{teo-loc-nlh}, and to show explicit examples of functions $\mu_1,\mu_2$ as in Example \ref{exemple}. However, the null controllability of \eqref{lin-intro} might be studied with less strict hypotheses than Assumption II.

\begin{itemize}
\item The result may be ensured when $\mu_1,\mu_2\in H^3(\T,\R)$ verify
 \begin{equation*}\begin{split}
\lag \mu_1,c_0 \rag_{L^2}\neq 0,\ \ \ \ \  \    \ \ \lag \mu_2,&c_0 \rag_{L^2}=0,\\
     \exists\,a_1,b_1>0,0<q_1 \leq 1 \,:\: e^{  b_1\lambda_k^{ q_1}} \left|\lag \mu_1 ,c_k\rag_{L^2}\right|\geq  a_1 ,\,\, &\text{ and }\,\, \lag \mu_1,s_k \rag_{L^2} = 0,\quad\forall\,k\in\N^*,\\
        \exists\,a_2,b_2>0,0<q_2 \leq 1 \,:\: e^{b_2\lambda_k^{ q_2}}\left|\lag \mu_2 ,s_k\rag_{L^2}\right|\geq  a_2 ,\,\, &\text{ and }\,\, \lag \mu_2,c_k \rag_{L^2}= 0,\quad\forall\,k\in\N^*.
\end{split}\end{equation*} 
It could be possible to adapt to our problem the theory developed in \cite{lissy} in the case of Dirichlet boundary conditions (see also \cite[Section 6.1]{ammar1}). Then, the strategy leading to Proposition \ref{prop-null} might ensure the null-controllability of \eqref{lin-intro} for every $T>T^*$ where $T^*$ is a minimal time which we expect to depend on $\max\{b_1,b_2\}$.
Note that in \cite[Theorem 1.2]{lissy} the upper bound of the control cost holds only for $T^*>0$. The relevant case for our purposes is however $T^*=0$, which seems to need further investigations.

\item Another approach to establish null controllability of \eqref{lin-intro} may be the following. For any $f:\T\rightarrow \R$, we introduce the decomposition $f=f^s +f^a$ where $$f^s(\cdot)=\frac{f(\cdot) +f(2\pi - \cdot)}{2},\ \ \ \ \ \ \ f^a(\cdot)=\frac{f(\cdot) -f(2\pi - \cdot)}{2}.$$ Now, $f^s$ is symmetric and $f^a$ is antisymmetric with respect to the point $x=\pi$. We apply such decomposition to 
\eqref{lin-intro} to obtain the following system
\begin{align}\label{lissy2_pre}
    \partial_t\begin{pmatrix}
           \xi^s\\
           \xi^a\\
    \end{pmatrix} 
    =   \begin{pmatrix}
           -\partial_x^2 + \kappa p & 0\\
          0 & -\partial_x^2 + \kappa p\\
    \end{pmatrix}
    \cdot
    \begin{pmatrix}
            \xi^s\\
           \xi^a\\
    \end{pmatrix}+\Phi\begin{pmatrix}
           \mu_1^s & \mu_2^s\\
          \mu_1^a & \mu_2^a\\
    \end{pmatrix}
    \cdot
    \begin{pmatrix}
            v_1\\
           v_2\\
    \end{pmatrix}.
      \end{align}
Problem \eqref{lissy2_pre} is now defined in $L^2_s(\T )\times L^2_a(\T )$, where $L^2_s(\T )$ and $L^2_a(\T )$ are the spaces of symmetric and antisymmetric $L^2$-functions, respectively. We uniquely identify both spaces $L^2_s(\T )$ and $L^2_a(\T )$ with $L^2( 0,\pi )$ and we note that, for any $\xi\in H^2(\T)$, it holds 
\begin{align}\label{obs}\xi^s|_{[0,\pi]}\in \{\psi\in H^2(0,\pi)\ :\ \partial_x\psi(0)=\partial_x\psi(\pi)=0\},\ \ \ \ \ \xi^a|_{[0,\pi]}\in H^2\cap H^1_0(0,\pi). \end{align}
Thus, \eqref{lissy2_pre} is equivalent to the following problem in $L^2(0,\pi)^2$
\begin{equation}\label{lissy2}
\begin{cases}
    \partial_t Y (t,x)
    =   A 
    \cdot Y +  \kappa\, p\, Y+\Phi\, B (x)
    \cdot
    V(t),& (t,x)\in (s_0,s_1)\times(0,\pi),\\
    Y(s_0)=Y_0,\\
    \end{cases}\end{equation}
      \begin{align*}
    Y=\begin{pmatrix}
           y^1\\
           y^2\\
    \end{pmatrix},\quad 
    A=   \begin{pmatrix}
           -\partial_x^2 & 0\\
          0 & -\partial_x^2\\
    \end{pmatrix},\ \ \ \ 
B=\begin{pmatrix}
           \mu_1^2 & \mu_2^2\\
          \mu_1^1 & \mu_2^1\\
    \end{pmatrix},\quad
V=
    \begin{pmatrix}
            v_1\\
           v_2\\
    \end{pmatrix},\quad
Y_0=\begin{pmatrix}
           y^1_0\\
           y^2_0\\
    \end{pmatrix},
      \end{align*}
where $\mu^1_j=\mu^s_j|_{[0,\pi]}$ and $\mu^2_j=\mu^a_j|_{[0,\pi]}$ with $j=1,2$, and $y_0^1=\xi_0^s$ and $y_0^2=\xi_0^a$. From \eqref{obs} it follows that problem \eqref{lissy2} is equipped with the following boundary conditions
\begin{align}\label{lissy2-cond}\partial_x y^1(0)=\partial_x y^1(\pi)=0,\ \ \ \ \ \ \ \ \  y^2(0)=y^2(\pi)=0. \end{align}
The null controllability property of \eqref{lin-intro} can therefore be studied by investigating the systems of parabolic equations \eqref{lissy2}-\eqref{lissy2-cond}. Observe that Assumption II implies that $\mu^2_1=\mu^1_2=0$ because $\mu_1$ is symmetric and $\mu_2$ is antisymmetric. Thus, the null-controllability of system \eqref{lissy2}-\eqref{lissy2-cond} boils down to control two decoupled equations with two different controls.
This kind of problem has been studied in the literature, mostly  in the case of an operator $B$ that does not depend on the space variable and/or the control $v\in L^2((s_0,s_1)\times(0,\pi),\R)$ depends on both time and space. 
In \cite{kalman1} for instance, the authors studied the problem by introducing an extension to partial differential equations of the classical Kalman condition (see also \cite{kalman2,kalman5,kalman3,kalman4}). From this perspective, the existing theory could potentially be expanded to encompass systems like \eqref{lissy2}-\eqref{lissy2-cond} and prove null controllability even without assuming Assumption II. Additionally, the techniques proposed in \cite{ammar1,assia2,system1} may prove applicable for this purpose.



    
\end{itemize}

\end{remark}

\subsection{Proof of Main Theorem B}
We now prove Theorem \ref{teo-loc-nlh}.

\begin{proof}[Proof of Theorem \ref{teo-loc-nlh}]

{$\mathbf (1)$ {\bf Time decomposition and preliminaries. }}
Let $T>0$ and define
\begin{equation*}
    T_f:=\min\{T,\frac{\pi^2}{6},\frac{\pi^2}{6}T_0\},
\end{equation*}
where $T_0$ is the constant in \eqref{estim-control-time}. Let $T_1$ be defined as follows
\begin{equation*}
    T_1:=\frac{6}{\pi^2}T_f.
\end{equation*}
Observe that, with this choice $0<T_1\leq1$. We now define the sequences
\begin{equation*}
    T_j:=\frac{T_1}{j^2},\quad j\geq1,\qquad\text{and}\qquad \tau_n:=\sum_{j=1}^n T_j,\quad n\geq0,
\end{equation*}
with the convention $\sum_{j=1}^0T_j=0$. It is easy to prove that
\begin{equation*}
    \sum_{j=1}^\infty T_j=T_f,
\end{equation*}
and thus we will perform an iterative control procedure on the consecutive time intervals $[\tau_n,\tau_{n+1}]$, $n\geq0$ so that at the limit $n\to \infty$ we will prove exact controllability of \eqref{0.1} to the ground state solution in time $T_f$.

\medskip

\noindent
{$\mathbf (2)$ {\bf Estimates in the first time step: inductive basis.}} Let us set $y=\psi-\Phi$, where $\psi$ and $\Phi$ are solutions of \eqref{0.1} associated with the initial conditions $\psi_0,\Phi\in H^3(\T)$ and controls $u\in H^1_{loc}(\R^+,\R^{q+2})$ and $\hat{u}:=(u_\kappa,0,0,\dots,0)$ with $u_\kappa=\frac{\kappa}{\Phi^p}$, respectively. Consider the equation satisfied by $y$ in a general time interval $(s_0,s_1)$:
\begin{equation}\label{y}
    \begin{cases}
        \partial_t y(t)-\partial^2_xy(t)+\kappa(y(t)+\Phi)^{p+1}=\langle u(t),Q\rangle y(t)+\langle u(t),Q \rangle \Phi,&t\in(s_0,s_1),\\
        y(s_0)=y_{s_0}=\psi(s_0)-\Phi.
    \end{cases}
\end{equation}
Our aim is to prove that system \eqref{y} is null controllable in time $T_f$ by means of a bilinear control $u$. This would imply that $\psi(T_f;\psi_0,0,u)=\Phi$. For this purpose, we shall first consider problem \eqref{y} and its linearization \eqref{lin-intro}, with initial condition $\xi_0=y_0:=\psi_0-\Phi$, in the time interval $[s_0,s_1]=[\tau_0,\tau_1]=[0,T_1]$. Observe that, thanks to the regularity of $\psi_0$ and $u$, the solution of \eqref{0.1} is globally in time well-defined (see Proposition \ref{thm-global-well-pos}), as well as the solution of the associated linear system  
\begin{equation}\label{xi-tau0-tau1}
\begin{cases}
\partial_t\xi(t)-\partial_x^2 \xi (t) +\kappa p\xi (t) = \lag v^1(t), Q\rag \Phi, &t\in(\tau_0,\tau_1),\\
\xi(\tau_0)=\xi_0,
\end{cases}
\end{equation}
(see Remark \ref{reg-sol-xi}).

From Proposition \ref{prop-null} we deduce that there exists a control, defined as $v^1=(0,\dots,0,v_1^1,v_2^1)$, with $(v_1^1,v_2^1)\in H^1_0((\tau_0,\tau_1),\R^2)$ such that 
\begin{equation*}
    \xi(\tau_1;\tau_0,\xi_0,v^1)=0,
\end{equation*}
and 
\begin{equation}\label{estim-v1}
    \norm{v^1}_{H^1(\tau_0,\tau_1)}\leq N(\tau_1)\norm{\xi_0}_{L^2}=N(\tau_1)\norm{y_0}_{L^2}\leq N(\tau_1)\norm{y_0}_{H^1} ,
\end{equation}
with $N(\tau_1)$ that satisfies \eqref{estim-control-time} (because $\tau_1\leq T_0$). We now define
\begin{equation*}
    u^1=(u_\kappa,0,\dots,0,v_1,v_2)\in H^1((\tau_0,\tau_1),\R^{q+2}).
\end{equation*}
Using such control in equation \eqref{0.1} (and so in \eqref{y} in the time interval $[\tau_0,\tau_1]$), one easily finds that \eqref{y} reads as
\begin{equation}\label{y-v1}
    \begin{cases}
        \partial_t y(t)-\partial^2_xy(t)+\kappa\displaystyle\sum_{j=2}^{p+1}\binom{p+1}{j}y^j(t)\Phi^{p+1-j}=\langle v^1(t),Q\rangle y(t)+\langle v^1(t),Q \rangle \Phi,&t\in(\tau_0,\tau_1),\\
        y(\tau_0)=y_0.
    \end{cases}
\end{equation}
We recall that $y_0\in H^3(\T,\R)$ and that $\mu_1,\mu_2\in H^3(\T,\R)$, thanks to Assumption II. Hence, from the definition of $y=\psi-\Phi$ and since $\psi$ and $\Phi$ are both solutions of \eqref{0.1}, we deduce 
from \emph{i.} of Proposition \ref{well-pos} that
\begin{equation}\label{estim-y-yo-v1}
   \sup_{\tau_0\leq t\leq \tau_1}\norm{y(t)}_{H^1 }= \sup_{\tau_0\leq t\leq \tau_1}\norm{\psi(t)-\Phi}_{H^1 }\leq C\left(\norm{\psi_0-\Phi}_{H^1 }+\norm{u^1-\hat u}_{L^2(\tau_1,\tau_1)}\right)\leq C\left(\norm{y_0}_{H^1 }+\norm{v^1}_{H^1(\tau_1,\tau_1)}\right).
\end{equation}
Thanks to \eqref{estim-v1}, we conclude that
\begin{equation*}
    \sup_{\tau_0\leq t\leq \tau_1}\norm{y(t)}_{H^1 }\leq C\left(1+N(\tau_1)\right)\norm{y_0}_{H^1 }.
\end{equation*}
Now, we introduce $w:=y-\xi$ (we use the control $u^1$ for $y$ and $v^1$ for $\xi$ on the time interval $[\tau_0,\tau_1]$) and we observe that $w$ solves the following problem
\begin{equation}\label{w}
    \begin{cases}
       \partial_t w(t)-\partial^2_x w(t)-\kappa(p+1)\Phi^pw(t)+\kappa\displaystyle\sum_{j=2}^{p+1}\binom{p+1}{j}\Phi^{p+1-j} y^j(t)=\langle v^1(t),Q\rangle y(t),&t\in(\tau_0,\tau_1),\\
        w(\tau_0)=0.
    \end{cases}
\end{equation}
Thanks to estimate \eqref{app-estim-w} of Proposition \ref{app-prop-w}, we have that
\begin{equation}\label{norm-w-tau0-tau1}
        \sup_{t\in[\tau_0,\tau_1]}\norm{w(t)}_{H^1  }\leq A_4(\tau_1,\norm{y_0}_{H^1 })\norm{y_0}^2_{H^1 }.
\end{equation}
Observe that our initial condition $\psi_0$ satisfies $\psi_0\in \{\psi\in H^3(\T,\R): \|\psi-\Phi\|_{H^1}<R_T\}$, with $R_T$ defined in \eqref{R_T}. Therefore, since $R_T<1$, we have that $\norm{y_0}_{H^1}<1$ and moreover 
\begin{equation*}
    N(\tau_1)\norm{y_0}_{H^1}\leq e^{\nu/\tau_1}e^{-6\Gamma_0/\tau_1}<1,
\end{equation*}
since $\Gamma_0>\nu$ (see definition \eqref{Gamma_0}). Hence, we obtain that
\begin{equation}\label{A4<K}
    A_4(\tau_1,\norm{y_0}_{H^1 })\leq K(\tau_1),
\end{equation}
    where
\begin{multline}\label{K}
    K^2(\tau):=2\left(2\kappa\tau(p+1)^2\sum_{j=2}^{p+1}\binom{p+1}{j}\Phi^{p+1-j}\left(1+N(\tau)^{4}\right)+C_Q^2N(\tau)^2\left(1+N(\tau)^2\right)\right)\cdot\\
    e^{\tau\left(2\kappa(p+1)\Phi^p+\kappa(p+1)\sum_{j=2}^{p+1}\Phi^{p+1-j}+1\right)}.
\end{multline}
\begin{remark}
    Note that for any $\tau\leq1$ it is possible to prove that
    \begin{equation}\label{estim-K}
        K(\tau)\leq e^{\Gamma_0/\tau}.
    \end{equation}
\end{remark}

Since $\tau_1\leq 1$, we deduce that \eqref{estim-K} holds for $\tau=\tau_1$. Using \eqref{A4<K} and \eqref{estim-K} (for $\tau=\tau_1$) in \eqref{norm-w-tau0-tau1} and recalling that $\xi(\tau_1;\tau_0,\xi_0,v^1)=0,$ we conclude that
\begin{equation}\label{induction_nee}
    \norm{y(\tau_1;\tau_0,y_0,u^1)}_{H^1}\leq e^{\Gamma_0/\tau_1}e^{-12\Gamma_0/\tau_1}=e^{-11\Gamma_0/\tau_1}<1.
\end{equation}
\smallskip

\noindent
{$\mathbf (3)$ {\bf Induction argument.}} Inequality \eqref{induction_nee} enables us to apply an iterative argument. In fact, we have just proved the first step of an induction procedure which consists in building in consecutive time intervals of the form $[\tau_{n-1},\tau_n]$, $n\geq1$, a control $u^n=(u_{\kappa},0,...,0,u_1^n,u_2^n)\in H^1((\tau_{n-1},\tau_n),\R^{q+2})$ such that 
\begin{equation}\label{induction}
    \begin{array}{ll}
        1.& \norm{u^n}_{H^1(\tau_{n-1},\tau_n)}\leq N(T_n)\norm{y_{n-1}}_{H^1 },\\
        2.& \xi(\tau_n;\tau_{n-1},y_{n-1},v^{n})=0,\\
        3.& \norm{y(\tau_n;\tau_{n-1},y_{n-1},u^{n})}_{H^1 }\leq e^{\left(\sum_{j=1}^n 2^{n-j}j^2-2^n6\right)\Gamma_0/T_1},\\
        4.& \norm{y(\tau_n;\tau_{n-1},y_{n-1},u^{n})}_{H^1 }\leq \prod_{j=1}^n K(T_j)^{2^{n-j}}\norm{y_0}^{2^n}_{H^1 },
    \end{array}
\end{equation}
where  $v^n:=u^n-\hat u$ and, thanks to Proposition \ref{thm-global-well-pos},
\begin{equation}\label{y_n-1}
    y_{n-1}:=y(\tau_{n-1},0,y_0,q^{n-1})\in H^3(\T,\R),
\end{equation}
\begin{equation}\label{q_n-1}
    q^{n-1}(t)=\sum_{j=1}^{n-1}u^j(t)\chi_{[\tau_{j-1},\tau_j]}(t)\qquad \text{(component-wise).}
\end{equation}
Observe that, by construction
\begin{equation*}
    y_n=y(\tau_{n};\tau_{n-1},y_{n-1},u^n)\in H^3(\T,\R),\qquad\forall\,n\geq1.
\end{equation*}
We emphasize that, due to the global well-posedness of the solution of \eqref{0.1} (and so of \eqref{y}), we do not meet any problem of existence when changing at each step of the proof the initial condition, the control and the time interval. 

Let us now prove the inductive step of the argument. We suppose properties 1.--4. of \eqref{induction} to hold for each $j=1,\dots,n-1$. Hence, assume that we have built controls $u^j=(u_\kappa,0,...,u_1^j,u_2^j)\in H^1((\tau_{j-1},\tau_j),\R^{q+2})$ such that 1.--4. of \eqref{induction} are satisfied. In particular, for $j=n-1$, suppose that
\begin{equation}\label{induction_n-1}
    \begin{array}{ll}
        1.& \norm{u^{n-1}}_{H^1(\tau_{n-2},\tau_{n-1})}\leq N(T_{n-1})\norm{y_{n-2}}_{H^1 },\\
        2.& \xi(\tau_{n-1};\tau_{n-2},y_{n-2},v^{n-1})=0,\\
        3.& \norm{y(\tau_{n-1};\tau_{n-2},y_{n-2},u^{n-1})}_{H^1 }\leq e^{\left(\sum_{j=1}^{n-1} 2^{n-1-j}j^2-2^{n-1}6\right)\Gamma_0/T_1},\\
        4.& \norm{y(\tau_{n-1};\tau_{n-2},y_{n-2},u^{n-1})}_{H^1 }\leq \prod_{j=1}^{n-1} K(T_j)^{2^{n-1-j}}\norm{y_0}^{2^{n-1}}_{H^1 }.
    \end{array}
\end{equation}
We now prove the existence of $u^n=(u_\kappa,0,...,0,u_1^n,u_2^n)\in H^1((\tau_{n-1},\tau_n),\R^{q+2})$ such that 1.--4. of \eqref{induction} are fulfilled. Define $q^{n-1}$ and $y_{n-1}$ as in \eqref{q_n-1} and \eqref{y_n-1}. 
We apply Proposition \ref{prop-null} to the linear system \eqref{lin-intro} with $(s_0,s_1)=(0,T_n)$ and $\xi(0)=y_{n-1}$. The result ensures the existence of a control $\tilde v^n=(0,...,0,\tilde u_1^n,\tilde u_2^n)\in H^1_0((0,T_n),\R^{q+2})$ that satisfies
\begin{equation*}
    \norm{\tilde v^n}_{H^1(0,T_n)}\leq N(T_n)\norm{y_{n-1}}_{H^1 },
\end{equation*}
with $N(T_n)\leq e^{\nu/T_n}$ and
\begin{equation*}
\xi(T_n;0,y_{n-1},\tilde v^n)=0.
\end{equation*}
We set $\tilde u^n(s):=\hat{u}+\tilde v^n(s)=(u_\kappa,0,...,0,\tilde u_1^n,\tilde u_2^n)$ with $s\in [0,T_n].$
Now, consider problem \eqref{y} for $(s_0,s_1)=(0,T_n)$ with $y(0)=y_{n-1}$. By using the control $\tilde u^n$ in \eqref{y}, we get
\eqref{y-v1} in $[0,T_n]$ with control $\tilde v^n$ instead of $v^1$. 
 %
We denote by $y(s;0,y_{n-1},\tilde u^n)$ with $s\in [0,T_n]$ its solution. 
Observe that, since $$    y_{n-1}=y(\tau_{n-1};0,y_0,q_{n-1})=y(\tau_{n-1};\tau_{n-2},y_{n-2},u^{n-1})\in H^3(\T,\R),$$
from 3. of \eqref{induction_n-1} we deduce that
\begin{equation}\label{N*y_n-1<1}
    N(T_n)\norm{y_{n-1}}_{H^1 }\leq e^{\nu n^2/T_1}e^{\left(\sum_{j=1}^{n-1} 2^{n-1-j}j^2-2^{n-1}6\right)\Gamma_0/T_1}\leq e^{-(2n+3)\Gamma_0/T_1}<1,
\end{equation}
where we have used that $\nu<\Gamma_0$ and the identity 
$$\sum_{j=0}^n \frac{j^2}{2^j}=2^{-n}(-n^2-4n+6(2^n-1)),\qquad\forall\,n\geq0.$$
Now, for $s\in[0, T_n]$, we define $$ w(s):=y(s;0,y_{n-1},\tilde u^n)- \xi(s;0,y_{n-1},\tilde v^n).$$
Thus, we apply Proposition \ref{app-prop-w}, and we deduce that
\begin{equation*}
    \norm{  y(T_n;0,y_{n-1},\tilde u^n)}_{H^1 }=\norm{ w(T_n;0,\tilde v^n)}_{H^1 }\leq A_4(T_n;\norm{y_{n-1}}_{H^1 })\norm{y_{n-1}}_{H^1 }^2.
\end{equation*}
We shift forward the time interval into $[\tau_{n-1},\tau_n]$ and we define
\begin{equation*}
    u^n(t):=\tilde{u}^n(t-\tau_{n-1}),\qquad v^n(t):=\tilde{v}^n(t-\tau_{n-1}),\qquad t\in(\tau_{n-1},\tau_n),
\end{equation*}
and we obtain that
\begin{equation*}
    \norm{u^n}_{H^1(\tau_{n-1},\tau_n)}\leq N(T_n)\norm{y_{n-1}}_{H^1 },
\end{equation*}
and
\begin{equation*}
    \xi(\tau_n;\tau_{n-1},y_{n-1},v^n)=\xi(T_n;0,y_{n-1},\tilde v^n)=0.
\end{equation*}
Thus, 1. and 2. of \eqref{induction} are fulfilled.

Recalling the definition of $K(\cdot)$ in \eqref{K}, the property \eqref{N*y_n-1<1} and that $\nu<\Gamma_0$, we deduce that
\begin{equation}\label{yKy_n-1}
    \norm{y(\tau_n;\tau_{n-1},y_{n-1},u^n)}_{H^1 }= \norm{y(T_n;0,y_{n-1},\tilde u^n)}_{H^1 }\leq K(T_n)\norm{y_{n-1}}^2_{H^1 }.
\end{equation}
Thus, using 3. of \eqref{induction_n-1} and the estimate above, we obtain
\begin{equation*}
    \norm{y(\tau_n;\tau_{n-1},y_{n-1},u^n)}_{H^1 }\leq e^{n^2  \Gamma_0/T_1}\left[e^{\left(\sum_{j=1}^{n-1} 2^{n-1-j}j^2-2^{n-1}6\right)\Gamma_0/T_1}\right]^2=e^{\left(\sum_{j=1}^{n} 2^{n-1-j}j^2-2^{n-1}6\right)\Gamma_0/T_1},
\end{equation*}
that is, 3. of \eqref{induction} is satisfied. Finally, using again \eqref{yKy_n-1} and thanks to 4. of \eqref{induction_n-1} we conclude that
\begin{equation*}
    \norm{y(\tau_n;\tau_{n-1},y_{n-1},u^n)}_{H^1 }\leq K(T_n)\left[\prod_{j=1}^{n-1} K(T_j)^{2^{n-1-j}}\norm{y_0}^{2^{n-1}}_{H^1 }\right]^2=\prod_{j=1}^{n} K(T_j)^{2^{n-j}}\norm{y_0}^{2^{n}}_{H^1 },
\end{equation*}
which is exactly 4. of \eqref{induction}. The induction argument is therefore concluded.

\medskip

\noindent
{$\mathbf (4)$ {\bf Conclusion.}} Let us now observe that, from 4. of \eqref{induction}, for every $n\in\N$ it holds that
\begin{equation*}
    \begin{split}
        \norm{y(\tau_n;\tau_{n-1},y_{n-1},u^n)}_{H^1 }&\leq \prod_{j=1}^{n} K(T_j)^{2^{n-j}}\norm{y_0}^{2^{n}}_{H^1 }\leq \prod_{j=1}^{n} \left(e^{\Gamma_0 j^2/T_1}\right)^{2^{n-j}}\norm{y_0}^{2^{n}}_{H^1 }\\
        &=e^{(\Gamma_0 2^n/T_1)\sum_{j=1}^n j^2/2^j}\norm{y_0}^{2^{n}}_{H^1 }\leq e^{(\Gamma_0 2^n/T_1)\sum_{j=1}^\infty j^2/2^j}\norm{y_0}^{2^{n}}_{H^1 }\\
        &\leq \left(e^{6\Gamma_0/T_1}\norm{y_0}_{H^1 }\right)^{2n},
    \end{split}
\end{equation*}
where in the last inequality we have used that $ \sum_{j=1}^\infty \frac{j^2}{2^j}=6.$ By definition, the above estimate reads as
\begin{equation*}
    \norm{y(\tau_n;0,y_{0},q^n)}_{H^1 }\leq \left(e^{6\Gamma_0/T_1}\norm{y_0}_{H^1 }\right)^{2n},
\end{equation*}
with 
$$q^n(t)=\sum_{j=1}^n u^j(t)\chi_{(\tau_{j-1},\tau_j)}(t)\qquad \text{(component-wise)}.$$
Taking the limit as $n\to+\infty$, we deduce that
\begin{equation*}
    \norm{\psi(\pi^2 T_1/6;0,\psi_0,q^\infty)-\Phi}_{H^1 }=\norm{y(\pi^2 T_1/6;0,\psi_0,q^\infty)}_{H^1 }=\norm{y(T_f;0,\psi_0,q^\infty)}_{H^1 }\leq 0,
\end{equation*}
thanks to the definition \eqref{R_T} of $R_T$. Thus, we construct a control $u\in H^1_{loc}(\R^+,\R^{q+2})$ 
\begin{equation*}
    u(t):=\begin{cases}
        
    \sum_{j=1}^\infty u^j(t)\chi_{(\tau_{j-1},\tau_j)}(t),&0<t\leq T_f, \\
    0 &t>T_f,
    \end{cases}\qquad\text{(component-wise)},
\end{equation*}
such that, at time $T_f\leq T$, the solution of \eqref{0.1} reaches exactly the ground state, that is, \begin{equation*}
    \psi(T_f;0,\psi^0,u)=\Phi.
\end{equation*}
Furthermore, we can derive a bound for the $H^1$-norm of the control
\begin{equation*}
    \norm{u}^2_{H^1(0,T_f)}\leq \sum_{n=1}^\infty\norm{u^n}^2_{H^1(0,T_f)}\leq \sum_{n=1}^\infty\left(N(T_n)\norm{y_{n-1}}_{H^1 }\right)^2\leq \sum_{n=1}^\infty e^{-2(2n+3)\Gamma_0/T_1}=\frac{e^{-\pi^2\Gamma_0/T_f}}{e^{2\pi^1\Gamma_0/(3T_f)}-1}.
\end{equation*}
where we have used \eqref{N*y_n-1<1}. Since $T_f\leq T$, we easily obtain \eqref{bound-control}.
\end{proof}

We are finally ready to prove Theorem \ref{mtb-gene} and Main Theorem B.

\begin{proof}[Proof of Theorem \ref{mtb-gene}]
The result is obtained by combining Theorem \ref{mta-gene} with $s=3$ and Theorem \ref{teo-loc-nlh}. In details, for every $\psi_0\in H^3(\T,\R)$ strictly positive, Theorem \ref{mta-gene} allows defining a dynamics steering $\psi_0$ close to $\Phi$ in the $H^3$ metric in an arbitrarily small time. If we are close enough to $\Phi$, we can use the local controllability result of Theorem \ref{teo-loc-nlh}. Thus, the solution of \eqref{0.1} reaches the target $\Phi$ in any finite time.
\end{proof}

\begin{proof}[Proof of Main Theorem B]
Main Theorem B is a direct consequence of Theorem \ref{mtb-gene} and the density property of $\mathcal{H}_\infty$ in $H^1(\T,\R)$ ensured by Proposition \ref{P:density}.
\end{proof}

\section*{Appendices}
\begin{appendices}
\section{Proof of Proposition \ref{well-pos}}\label{AppA}

In this section, we shall prove the existence and uniqueness of solutions  of \eqref{0.1}.

\begin{proof}{$\mathbf (1)$ {\bf Existence and uniqueness of solutions.}}     Let $s>d/2$ and $\psi_0\in H^s(\T^d,\R)$. For the sake of shortness, we consider $\kappa=1$. However, the proof remains valid in the general case. We are going to show that there exists $t_1>0$ such that the Cauchy problem \eqref{0.1} admits a unique solution $\psi\in C^0([0,t_1],H^s(\T^d,\R))$.  We define the following quantities
    \begin{equation*}
        M:=\sup\{\norm{e^{t\Delta}-I}_{\mathcal{L}(H^s(\T^d,\R))},\,0\leq t\leq 1\},\ \ \ \  \ \ \        r(\psi_0):=2M\norm{\psi_0}_{H^s},
    \end{equation*}
    \begin{equation*}
        C(Q):=\max_{1\leq i \leq d }\norm{Q_i}_{H^s},\qquad C(d,u,Q):=\sqrt{2d}\norm{u}_{L^2(0,1)}C(Q).
    \end{equation*}
    Observe that, since $s>d/2$, we deduce that the embedding $H^s(\T^d,\R)\hookrightarrow C^0(\T^d,\R) $ is continuous, that is, there exists a constant $C(\T^d)$ such that
    \begin{equation}
        \sup_{x\in\T^d}|y(x)|\leq C(\T^d)\norm{y}_{H^s}, \quad \forall\,y\in H^s(\T^d,\R).
    \end{equation}
    We now define
    \begin{equation}
        \delta:=\min\left\{1,\frac{(r(\psi_0))^2}{4C(\T^d)^2\left(C(d,u,Q)(r(\psi_0)+\norm{\psi_0}_{H^s})+2^p((r(\psi_0))^{p+1}+\norm{\psi_0}^{p+1}_{H^s})\right)^2}\right\},
    \end{equation}
    and set $t_1=\delta$.
    
    We denote $B:=B_{C^0([0,t_1],H^s))}(\psi_0,r(\psi_0))$ the ball in the space $C^0([0,t_1],H^s(\T^d,\R))$ of center $\psi_0$ and radius $r(\psi_0)$. For every $\psi\in B$ we define the following function
    \begin{equation}
        \Phi(\psi)(t):=e^{t\Delta}\psi_0+\int_0^te^{(t-s)\Delta}\left(\langle u(s), Q(x)\rangle \psi(s,x)+ \psi(s,x)^{p+1}\right)ds. 
    \end{equation}
    Let us show that $\Phi$ maps $B$ into itself:
    \begin{equation*}
        \begin{split}
            \norm{\Phi(\psi)(t)-\psi_0}_{H^s }&\leq \norm{e^{t\Delta}\psi_0-\psi_0}_{H^s }\quad+\norm{\int_0^te^{(t-s)\Delta}\left(\langle u(s), Q\rangle \psi(s)+\psi(s)^{p+1}\right)ds}_{H^s }\\
            &\leq M\norm{\psi_0}_{H^s }+\int_0^t\norm{\langle u(s), Q\rangle\psi(s)}_{H^s }+\norm{\psi(s)^{p+1}}_{H^s }ds\\
            &\leq M\norm{\psi_0}_{H^s }+C(\T^d)\left(C(Q)\int_0^t \sum_{i=1}^q|u_i(s)|\norm{\psi(s)}_{H^s }ds+\int_0^t\norm{\psi(s)}_{H^s }^{p+1}ds\right)\\
            &\leq M\norm{\psi_0}_{H^s } +C(\T^d)\sqrt{d}C(Q)\left(\int_0^t \sum_{i=1}^q|u_i(s)|^2ds\right)^{1/2}\left(2\int_0^t\norm{\psi(s)-\psi_0}^2_{H^s }+\norm{\psi_0}^2_{H^s }ds\right)^{1/2}\\
            &\quad\quad+C(\T^d)2^p\int_0^t\norm{\psi(s)-\psi_0}_{H^s }^{p+1}+\norm{\psi_0}^{p+1}_{H^s }ds.
        \end{split}
    \end{equation*}
    Thus, we deduce that $\Phi(\psi)\in B$ since
    \begin{multline*}
        \sup_{t\in[0,t_1]}\norm{\Phi(\psi)(t)-\psi_0}_{H^s }\leq \frac{r(\psi_0)}{2}+C(\T^d)\sqrt{2d}C(Q)\norm{u}_{L^2(0,1)}\left(\sup_{t\in[0,t_1]}\norm{\psi(t)-\psi_0}_{H^s }+\norm{\psi_0}_{H^s }\right)\sqrt{t_1}\\
        +C(\T^d)2^p\left(\sup_{t\in[0,t_1]}\norm{\psi(t)-\psi_0}_{H^s }^{p+1}+\norm{\psi_0}^{p+1}_{H^s }\right)\sqrt{t_1}\leq \frac{r(\psi_0)}{2}+\frac{r(\psi_0)}{2}=r(\psi_0).
    \end{multline*}
    Now, we show that $\Phi$ is a contraction over $B$. Let $\psi,\phi\in B$, then
    \begin{equation*}
        \begin{split}
            \norm{\Phi(\psi)(t)-\Phi(\phi)(t)}_{H^s}&=\norm{\int_0^t \langle u(s),Q(x)\rangle(\psi(s)-\phi(s))+\psi(s)^{p+1}-\phi(s)^{p+1}ds}_{H^s}\\
            &\leq \int_0^t \norm{\langle u(s),Q(x)\rangle(\psi(s)-\phi(s))}_{H^s}ds\\
            &\quad+C(\T^d)\int_0^t\norm{\psi(s)-\phi(s)}_{H^s}\sum_{j=0}^p\norm{\psi(s)}^j_{H^s}\norm{\phi(s)}^{p-j}_{H^s} ds\\
            &\leq C(Q)\sqrt{2d}\norm{u}_{L^2((0,1),\R^{q+2})}\left(\int_0^t\norm{\psi(s)-\phi(s)}^2_{H^s}ds\right)^{1/2}\\
            &\quad+C(\T^d)2^{p-2}\int_0^t\norm{\psi(s)-\phi(s)}_{H^s}L(\psi,\phi)ds,
        \end{split}
    \end{equation*}
where
\begin{equation*}
    L(\psi,\phi):=\sum_{j=0}^p\left(\norm{\psi(s)-\psi_0}^j_{H^s}+\norm{\psi_0}^j_{H^s}\right)\left(\norm{\phi(s)-\psi_0}^{p-j}_{H^s}+\norm{\psi_0}^{p-j}_{H^s}\right) .
\end{equation*}
    Therefore, we get that
    \begin{equation*}
        \begin{split}
            \sup_{0\leq t\leq t_1}\norm{\Phi(\psi)(t)-\Phi(\phi)(t)}_{H^s}&\leq C(Q)\sqrt{2d}\norm{u}_{L^2((0,1)}\sqrt{t_1}\sup_{0\leq t\leq t_1}\norm{\psi(t)-\phi(t)}_{H^s}\\
            &\quad+C(\T^d)2^{p-2}\sqrt{t_1}\sup_{0\leq t \leq t_1}L(\psi,\phi)\sup_{0\leq t \leq t_1}\norm{\psi(t)-\phi(t)}_{H^s}\\
            &\leq \left(C(d,u,Q)+C(\T^d)2^{p-2}\tilde L(\psi_0)\right)\sqrt{t_1}\sup_{0\leq t \leq t_1}\norm{\psi(t)-\phi(t)}_{H^s},
        \end{split}
    \end{equation*}
where 
\begin{equation*}
    \tilde L(\psi_0):=\sum_{j=0}^p \left(r(\psi_0)^j+\norm{\psi_0}^j_{H^s}\right)\left(r(\psi_0)^{p-j}+\norm{\psi_0}^{p-j}_{H^s}\right).
\end{equation*}
With the same kind of computation one can prove that
\begin{equation*}
        \sup_{0\leq t\leq t_1}\norm{\Phi^2(\psi)(t)-\Phi^2(\phi)(t)}_{H^s}\leq \left(C(d,u,Q)+C(\T^d)2^{p-2}\tilde L(\psi_0)\right)^2\frac{(\sqrt{t_1})^2}{\sqrt{2}}\sup_{0\leq t\leq t_1}\norm{\psi(t)-\phi(t)}_{H^s},
\end{equation*} 
and, iterating the procedure, one shows that
\begin{equation*}
        \sup_{0\leq t\leq t_1}\norm{\Phi^n(\psi)(t)-\Phi^n(\phi)(t)}_{H^s}\leq\left(C(d,u,Q)+C(\T^d)2^{p-2}\tilde L(\psi_0)\right)^n\frac{(\sqrt{t_1})^n}{\sqrt{n!}}\sup_{0\leq t\leq t_1}\norm{\psi(t)-\phi(t)}_{H^s}.
\end{equation*} 
For $n$ large enough, it holds that $$\left(C(d,u,Q)+C(\T^d)2^{p-2}\tilde L(\psi_0)\right)^n\frac{(\sqrt{t_1})^n}{\sqrt{n!}}<1.$$
Hence, form a well-known corollary of the Banach fixed point Theorem, we deduce that $\Phi$ is a contraction over $B$ and so there exists a unique fixed point $\psi\in B$ which is the solution of \eqref{0.1}. Furthermore, it holds that 

\begin{equation*}
    \sup_{t\in[0,t_1]}\norm{\psi(t)}_{H^s}\leq (2M+1)\norm{\psi_0}_{H^s}.
\end{equation*}
We have just proved that if $\psi$ is solution  of \eqref{0.1} in $[0,\tau]$, then we can extend it into $[0,\tau+\delta(\tau)]$. Indeed, by defining the quantities 
$$M(\tau):=\sup_{\tau\leq t\leq \tau + 1}\left\{\norm{e^{t\Delta}-I}_{\mathcal{L}(H^s(\T^d,\R))}\right\} ,\quad  r(\tau,\psi(\tau)):=2M(\tau)\norm{\psi(\tau)}_{H^s},\quad C(d,u,Q):=\sqrt{2d}\norm{u}_{L^2(\tau,\tau+1)}C(Q),$$
and
\begin{equation*}
    \delta(\tau):=\min\left\{1,\frac{(r(\psi(\tau))^2}{4C(\T^d)^2\left(C(d,u,Q)(r(\psi(\tau))+\norm{\psi(\tau)}_{H^s})+2^p((r(\psi(\tau)))^{p+1}+\norm{\psi(\tau)}^{p+1}_{H^s})\right)^2}\right\},
\end{equation*}
    then one just sets $\psi(t)=\zeta(t)$ for $t\in[\tau,\tau+\delta(\tau)]$, where
    \begin{equation*}
        \zeta(t)=e^{(t-\tau)\Delta}\psi(\tau)+\int_\tau^te^{(t-s)\Delta}\left(\langle u(s), Q(x)\rangle \psi(s,x)+ \psi(s,x)^{p+1 }\right)ds. 
    \end{equation*}
    Let $[0,\TT)$ be the maximal interval of existence of $\psi$, solution  of the \eqref{0.1}. If $\TT<+\infty$, then $\norm{\psi(t)}_{H^s}\to+\infty$ as $t\to\TT^-$, otherwise $\psi$ could be extended, which contradicts the maximality of $\TT$. Observe that, for any $0<T<\TT$,
    \begin{equation}\label{eq-bound-solution}
        \sup_{t\in[0,T]}\norm{\psi(t)}_{H^s}\leq C \norm{\psi_0}_{H^s}.
    \end{equation}

\smallskip

\noindent
{$\mathbf (2)$ {\bf Proof of the continuity} \eqref{psi-phi}  {\bf and the stability} \eqref{stability}.} Let $\psi,\phi\in C^0([0,T],H^s(\T^d,\R))$, with $0\leq T\leq \min\{\TT(\psi_0),\TT(\phi_0)\}$, be the solutions  of the \eqref{0.1} corresponding to the initial conditions $\psi_0$ and $\phi_0$ and controls $u$ and $v$, respectively. Then, 
    \begin{equation*}
        \begin{split}
            \norm{\psi(t)-\phi(t)}_{H^s}&\leq \norm{\psi_0-\phi_0}_{H^s}\\
            &\quad+\int_0^t(\norm{\langle u(s),Q\rangle\psi(s)-\langle v(t),Q\rangle \phi(s)}_{H^s}+\norm{\psi(s)^{p+1}-\phi(s)^{p+1}}_{H^s})ds\\
            &\leq \norm{\psi_0-\phi_0}_{H^s}+ \int_0^t \norm{\langle u(s),Q\rangle (\psi(s)-\phi(s))}_{H^s}ds+\\
            &\quad+\int_0^t \norm{(\langle u(s),Q\rangle-\langle v(s),Q\rangle )\phi(s)}_{H^s}ds\\
            &\quad+C(\T^d)\int_0^t\norm{\psi(s)-\phi(s)}_{H^s}\sum_{j=0}^p\norm{\psi(s)}^j_{H^s}\norm{\phi(s)}^{p-j}_{H^s} ds\\
            &\leq \norm{\psi_0-\phi_0}_{H^s} +C(Q)\sqrt{2d}\norm{u}_{L^2(0,t)}\sqrt{t}\sup_{0\leq s\leq t}\norm{\psi(s)-\phi(s)}_{H^s}\\
            &\quad+ C(Q)\sqrt{2d}\norm{u-v}_{L^2(0,t)}\sqrt{t}\sup_{0\leq s \leq t}\norm{\phi(s)}_{H^s}\\
            &\quad+C(\T^d)\sqrt{t}\left(\sum_{j=0}^p\sup_{0\leq s \leq t}\norm{\psi(s)}^j_{H^s}\sup_{0\leq s \leq t}\norm{\phi(s)}^{p-j}_{H^s}\right) \sup_{0\leq s \leq t} \norm{\psi(s)-\phi(s)}_{H^s}.
        \end{split}
    \end{equation*}
    Therefore, we get the existence of $C_1,C_2,C_3>0$, which depend only on the parameters of the problem, such that
    \begin{equation*}
        \begin{split}
            \sup_{0\leq t \leq T}\norm{\psi(t)-\phi(t)}_{H^s}&\leq \norm{\psi_0-\phi_0}_{H^s}+ \left(C(Q)\sqrt{2d}\sqrt{T}\sup_{0\leq t \leq T}\norm{\phi(t)}_{H^s}\right)\norm{u-v}_{L^2(0,T)}\\
            &\quad+\left(C(d,u,Q)+C(\T^d)\sum_{j=0}^p\sup_{0\leq t \leq T}\norm{\psi(t)}^j_{H^s}\sup_{0\leq t \leq T}\norm{\phi(t)}^{p-j}_{H^s}\right)\sqrt{T} \sup_{0\leq t \leq T} \norm{\psi(t)-\phi(t)}_{H^s}\\
            &\leq \norm{\psi_0-\phi_0}_{H^s}+C_1R\sqrt{T}\norm{u-v}_{L^2(0,T)}+(C_2+C_3R^p)\sqrt{T}\sup_{0\leq t \leq T} \norm{\psi(t)-\phi(t)}_{H^s},
        \end{split}
    \end{equation*}
    where we have used \eqref{eq-bound-solution} and that $\psi_0,\phi_0\in B_{H^s}(0,R)$. If $(C_2+C_3R^p)\sqrt{T}<1$, then we obtain the validity of \eqref{psi-phi} and then 
    \begin{equation*}
        \sup_{0\leq t \leq T}\norm{\psi(t)-\phi(t)}_{H^s}\leq C\left( \norm{\psi_0-\phi_0}_{H^s}+\norm{u-v}_{L^2(0,T)}\right).
    \end{equation*}
    Otherwise, we can subdivide the interval $[0,T]$ into subintervals where inequality $(C_2+C_3R^p)\sqrt{T}<1$ holds and obtain the result. Finally, thanks to \eqref{psi-phi} it is possible to prove \eqref{stability} (see also the proof of \cite[Proposition 2.1]{DN-local-exact}).
\end{proof}

\section{Global well-posedness and proof of Proposition \ref{thm-global-well-pos}}\label{AppB}

This appendix aims to prove Proposition \ref{thm-global-well-pos}. We show that for more regular initial conditions, $\psi_0\in H^3(\T,\R)$, the solution of the one-dimensional problem ($d=1$) \eqref{0.1} is global in time, that is, the maximum time of existence is $\TT(\psi_0)=+\infty$. First, we observe that if $\psi_0\in H^1(\T,\R)$ then, from Proposition \ref{well-pos}, there exists $\TT(\psi_0)$ such that for every $0<T <\TT(\psi_0)$ there exists a unique solution $\psi\in C^0([0,T],H^1(\T,\R))$. Second, if $\mathcal{T}(\psi_0)=+\infty$ then the solution is global in time, otherwise
\begin{equation*}
\norm{\psi(t)}_{H^1 }\to +\infty\qquad\text{as }t\to\mathcal{T}(\psi_0).
\end{equation*}
It can be proved, thanks to the analyticity of the semigroup generated by $\Delta$ and applying \cite[Theorem 3.1, p. 143]{DaPrato} to the fixed point argument used in the proof of Proposition \ref{well-pos}, that
\begin{equation*}
    \psi\in H^1((0,T),L^2 (\T,\R))\cap L^2((0,T),H^2(\T,\R) ).
\end{equation*}
Let us consider $\Psi:=\psi_t$ that formally solves the problem
\begin{equation}\label{eq-Psi}
\begin{cases}
\partial_t \Psi-\Delta\Psi+\kappa(p+1)\psi^{p}\Psi=\langle u'(t),Q(x)\rangle\psi+\langle u(t),Q(x)\rangle \Psi\\
\Psi(0)=\Psi_0:=\Delta \psi_0-\kappa\,\psi_0^{p+1}+\langle u(0),Q(x)\rangle \psi_0.
\end{cases}
\end{equation}
Let us now prove the following result which is necessary for the proof of Proposition \ref{thm-global-well-pos}.

\begin{proposition}\label{prop-reg-psi-t}
Let $\psi_0\in H^3(\T,\R)$, $Q\in H^3(\T,\R^{q+2})$, $u\in H^1_{\text{loc}}((0,+\infty),\R^{q+2})$, $0<T<\TT(\psi_0,u)$ and $\kappa\geq0$. Then, there exists a unique solution $\Psi\in L^2((0,T),H^1(\T,\R))$ of \eqref{eq-Psi} given by
\begin{equation*}
\Psi(t)=e^{t\Delta}\Psi_0+\int_0^t e^{(t-s)\Delta}\left(\langle u'(s),Q(x)\rangle\psi(s)+\langle u(s),Q(x)\rangle \Psi(s)-\kappa(p+1)\psi(s)^{p}\Psi(s)\right)ds,
\end{equation*} 
where $\psi$ is the unique solution of \eqref{0.1} with initial condition $\psi_0$. Furthermore, it holds that $\Psi=\psi_t$.
\end{proposition}

\begin{proof}{$\mathbf (1)$ {\bf Existence and uniqueness of solutions.}} We already know from the local well-posedness of \eqref{0.1} that $\Psi=\psi_t\in L^2((0,T),L^2(\T,\R))$. We now consider equation \eqref{eq-Psi}. In order to apply a fix point argument, for every $\xi\in C([0,T],H^1(\T,\R))$, we consider the following map
\begin{equation*}
\Phi(\xi)(t):=e^{t\Delta}\Psi_0+\int_0^t e^{(t-s)\Delta}\Big(\langle u'(s),Q(x)\rangle\psi(s)+\langle u(s),Q(x)\rangle \xi(s)-\kappa(p+1)\psi(s)^{p}\xi(s)\Big)ds,
\end{equation*}
where $\psi\in H^1((0,T),L^2(\T,\R))\cap L^2((0,T),H^2(\T,\R))\cap C([0,T],H^3(\T,\R))$ is the solution of \eqref{0.1} with initial condition $\psi_0\in H^3(\T,\R)$. We first observe that $\Psi_0=\Delta \psi_0-\kappa \psi_0^p\psi_0+\langle u(0),Q(x)\rangle \psi_0$ is well-defined and is in $H^1(\T,\R)$. Let us first prove that $\Phi$ maps  $C([0,T],H^1(\T,\R))$ into itself. Since 
\begin{equation*}
f(\cdot):=\langle u'(\cdot),Q\rangle\psi+\langle u(\cdot),Q\rangle \xi(\cdot)-\kappa(p+1) \psi(\cdot)^{p}\xi(\cdot)\in L^2((0,T),L^2(\T,\R)),
\end{equation*}
from \cite[Theorem 3.1, p. 143]{DaPrato} we deduce that $\Phi(\xi)\in  H^1([0,T],L^2(\T,\R))\cap L^2((0,T),H^2(\T,\R))$ for every $\xi\in C([0,T],H^1(\T,\R))$. Thus, we deduce from \cite[Proposition 2.1, p. 22 and Theorem 3.1, p. 23]{Lions} that $\Phi$ maps $C([0,T],H^1(\T,\R))$ into itself. We now prove that $\Phi$ is a contraction. Let $\xi,\tilde{\xi}\in C([0,T],H^1(\T,\R))$. Then,
\begin{equation*}
\begin{split}
&\sup_{t\in[0,T]}\norm{\Phi(\xi)(t)-\Phi(\tT{\xi})}_{H^1}=\sup_{t\in[0,T]}\norm{\int_0^t e^{(t-s)\Delta}\left(\langle u(s),Q(x)\rangle (\xi(s)-\tT{\xi}(s))-\kappa(p+1)\psi(s)^{p}(\xi(s)-\tT{\xi}(s))\right)ds}_{H^1}\\
&\quad\leq C\norm{\left(\langle u(\cdot),Q\rangle-\kappa(p+1)\psi(s)^{p}\right) (\xi(\cdot)-\tT{\xi}(\cdot))}_{L^2((0,T),L^2(\T))}\leq C\left(\int_0^t\norm{\xi(s)-\tT{\xi}(s)}_{L^2}^2ds\right)^{1/2}\\
&\quad\leq C\sqrt{T}\sup_{t\in[0,T]}\norm{\xi(t)-\tT{\xi}(t)}_{L^2}\leq C\sqrt{T}\sup_{t\in[0,T]}\norm{\xi(t)-\tT{\xi}(t)}_{H^1}.
\end{split}
\end{equation*}
If $C\sqrt{T}<1$ then $\Phi$ is a contraction. Otherwise, one can divide the interval $[0,T]$ in a finite number of sub-intervals where $\Phi$ is a contraction and conclude the argument. So, we deduce that $\Phi$ admits a unique fix point $\Psi$ in the space $C([0,T],H^1(\T,\R))$ which is the solution of \eqref{eq-Psi}. 

\smallskip

{$\mathbf (1)$ {\bf Proof of the identity $\Psi=\psi_t$.}} We now prove that the unique solution of \eqref{eq-Psi} is indeed $\psi_t$. Let $t\in[0,T]$, $\tau>0$ such that $0<t+\tau<T$, and consider the difference
\begin{equation*}
\frac{\psi(t+\tau)-\psi(t)}{\tau}-\Psi(t).
\end{equation*}
Our aim is to prove that the above quantity converges to $0$ as $\tau\to0$. Thanks to the expression of $\psi$ and $\Psi$ as mild solutions, we get
\begin{equation*}
\begin{split}
&\frac{\psi(t+\tau)-\psi(t)}{\tau}-\Psi(t)=\frac{1}{\tau}\left\{e^{(t+\tau)\Delta}\psi_0-e^{t\Delta}\psi_0\right\}+\frac{1}{\tau}\int_0^\tau e^{(t+\tau-s)\Delta}\left[-\kappa\,\psi(s)^{p+1}+\langle u(s),Q\rangle\psi(s)\right]ds\\
&\qquad+\frac{1}{\tau}\int_0^t e^{(t-s)\Delta}\left[-\kappa\,\psi(s+\tau)^{p+1}+\langle u(s+\tau),Q\rangle\psi(s+\tau)+\kappa\,\psi(s)^{p+1}-\langle u(s),Q\rangle\psi(s)\right]ds\\
&\qquad-e^{t\Delta}(\Delta \psi_0-\kappa\,\psi_0^{p+1}+\langle u(0),Q\rangle\psi_0)-\int_0^t e^{(t-s)\Delta}\left(\langle u'(s),Q(x)\rangle\psi(s)+\langle u(s),Q(x)\rangle \xi(s)-\kappa(p+1)\psi(s)^{p}\xi(s)\right)ds\\
&\quad=\frac{1}{\tau}\left\{e^{(t+\tau)\Delta}\psi_0-e^{t\Delta}\psi_0\right\}-e^{t\Delta}\psi_0-e^{t\Delta}(-\kappa\,\psi_0^{p+1}+\langle u(0),Q\rangle\psi_0))\\
&\qquad+\frac{1}{\tau}\int_0^\tau e^{(t+\tau-s)\Delta}\left[-\kappa\,\psi(s)^{p+1}+\langle u(s),Q\rangle\psi(s)\right]ds+\int_0^te^{(t-s)\Delta}\left[-\kappa\,\psi(s+\tau)^p\left(\frac{\psi(s+\tau)-\psi(s)}{\tau}-\Psi(s)\right)\right]ds\\
&\qquad+\int_0^te^{(t-s)\Delta}\left[-\kappa\left[p\,\psi(s)^{p } \left(\frac{\psi(s+\tau)-\psi(s)}{\tau}-\Psi(s)\right)+o\left(\frac{\psi(s+\tau)-\psi(s)}{\tau}\right)\right]\right]ds\\
&\qquad+\int_0^t e^{(t-s)\Delta}\left[\langle \frac{u(s+\tau)-u(s)}{\tau},Q\rangle \psi(s+\tau)+\langle u(s+\tau),Q\rangle \left(\frac{\psi(s+\tau)-\psi(s)}{\tau}-\Psi(s)\right)\right]ds\\
&\qquad+\int_0^t e^{(t-s)\Delta}\left[-\kappa\,\psi(s+\tau)^p\Psi(s)-p\kappa\,\psi(s)^p\Psi(s)+\langle u(s+\tau),Q\rangle\Psi(s)\right]ds\\
&\qquad-\int_0^t e^{(t-s)\Delta}\left(\langle u'(s),Q(x)\rangle\psi(s)+\langle u(s),Q(x)\rangle \Psi(s)-\kappa(p+1)\psi(s)^{p}\Psi(s)\right)ds.
\end{split}
\end{equation*}
Now, by taking the absolute value of the above identity, we obtain
\begin{equation*}
\begin{split}
&\left|\frac{\psi(t+\tau)-\psi(t)}{\tau}-\Psi(t)\right|\leq \left|\frac{1}{\tau}\left\{e^{(t+\tau)\Delta}\psi_0-e^{t\Delta}\psi_0\right\}-e^{t\Delta}\psi_0\right|\\
&\qquad +\left|\frac{1}{\tau}\int_0^\tau e^{(t+\tau-s)\Delta}\left[-\kappa \psi(s)^{p+1}+\langle u(s),Q\rangle\psi(s)\right]ds-e^{t\Delta}(-\kappa\,\psi_0^{p+1}+\langle u(0),Q\rangle\psi_0))\right|\\
&\qquad+\int_0^t\Big|e^{(t-s)\Delta}\Big[\langle \frac{u(s+\tau)-u(s)}{\tau},Q\rangle \psi(s+\tau)-\langle u'(s),Q(x)\rangle\psi(s)-\kappa\,\psi(s+\tau)^p\Psi(s)\\
&\qquad\qquad-p\kappa\,\psi(s)^p\Psi(s)+\langle u(s+\tau),Q\rangle\Psi(s)-\langle u(s),Q(x)\rangle \Psi(s)+\kappa(p+1)\psi(s)^{p}\Psi(s)\Big]\Big|ds\\
&\qquad +\int_0^t \left|-\kappa\,\psi(s+\tau)^p-\kappa p\,\psi(s)^p+\langle u(s+\tau),Q\rangle\right|\left|\frac{\psi(s+\tau)-\psi(s)}{\tau}-\Psi(s)\right|ds\\
&\qquad+\int_0^t\left|e^{(t-s)\Delta}o\left(\frac{\psi(s+\tau)-\psi(s)}{\tau}\right)\right|ds,\\
\end{split}
\end{equation*}
which is of the type $f(t)\leq\alpha(t)+\int_0^t \beta(s)f(s)ds.$ We apply the Gr\"{o}nwall's Lemma, and we deduce that $
f(t)\leq \alpha(t)+\int_0^t\alpha(s)\beta(s)e^{\int_s^t\beta(r)dr}ds,$ and then, by taking the limit as $\tau\to 0$, we finally obtain that
\begin{equation*}
\left|\frac{\psi(t+\tau)-\psi(t)}{\tau}-\Psi(t)\right|\to 0. 
\end{equation*}
Hence, we have proved that the unique solution $\Psi$ of \eqref{eq-Psi} coincides with $\psi_t$.
\end{proof}

We are now ready to prove the global well-posedness of \eqref{0.1} with $d=1$ stated in Proposition \ref{thm-global-well-pos}.

\begin{proof}[Proof of Proposition \ref{thm-global-well-pos}]
From the local well-posedness result, we know that for any $\psi_0\in H^1(\T,\R)$ there exists $\TT (\psi_0,u)>0$ such that for any $0<t< \TT(\psi_0,u)$ there exists a unique solution $\psi\in H^1((0,t),L^2(\T,\R))\cap L^2((0,t),H^2(\T,\R))\cap C([0,t],H^1(\T,\R))$ of problem \eqref{0.1}. Furthermore, if $\TT(\psi_0,u)<+\infty$, then $\norm{\psi(t)}_{H^1}\to +\infty$ as $t\to\TT(\psi_0,u).$ We shall prove that
\begin{equation*}
\norm{\psi(t)}_{H^1}\leq C\qquad\text{as }t\to\TT(\psi_0,u),
\end{equation*}
and we would deduce that $\TT(\psi_0,u)=+\infty$, that is, the solution is globally well-defined. We recall that $\norm{\psi}^2_{H^1}:=\norm{\psi}_{L^2}^2+\norm{\partial_x \psi}_{L^2}^2.$ For almost every $0<t<\mathcal{T}(\psi_0,u)$, we multiply the equation in \eqref{0.1} by $\psi$ and we obtain
\begin{equation*}
\langle \partial_t \psi,\psi\rangle_{L^2}-\langle\Delta\psi,\psi\rangle_{L^2}+\kappa\langle \psi^{p+1},\psi\rangle_{L^2}= \big\langle\langle u(t),Q(x)\rangle\psi,\psi\big\rangle_{L^2}.
\end{equation*}
Recalling that $\kappa\geq0$ and $p\in 2\N$, we get $\kappa\langle \psi^{p+1},\psi\rangle_{L^2}\geq0$ and then $ \frac{1}{2}\frac{d}{dt}\norm{\psi(t)}_{L^2}^2\leq C(t) \norm{\psi(t)}^2_{L^2}$ thanks to the accretivity of $-\Delta$. Therefore, for a.e. $t\in (0,\TT(\psi_0,u))$, we have the inequality
\begin{equation}\label{eq-norm-2-psi}
\norm{\psi(t)}^2_{L^2}\leq \norm{\psi_0}^2_{L^2}e^{2\int_0^t C(s)ds}\leq \norm{\psi_0}^2_{L^2}e^{2C\sqrt{\TT(\psi_0,u)}\norm{u}_{L^2(0,\TT(\psi_0,u))}}. 
\end{equation}
Now we multiply equation \eqref{0.1} by $-\Delta\psi$ and we get 
\begin{equation}\label{eqB3}
-\langle \partial_t \psi,\Delta\psi\rangle_{L^2}+\langle\Delta\psi,\Delta\psi\rangle_{L^2}-\kappa\langle\psi^{p+1},\Delta\psi\rangle_{L^2}=- \big\langle\langle u(t),Q(x)\rangle\psi,\Delta\psi\big\rangle_{L^2}.
\end{equation}
Let us observe that thanks to the properties of $\Delta$, for every $\xi \in H^2(\T)$, it holds that
    \begin{equation*}
         t\mapsto \langle\Delta\xi(t),\xi(t)\rangle_{L^2}
    \end{equation*}
    is absolutely continuous and
    \begin{equation*}
         \frac{d}{dt}\langle \Delta\xi(t),\xi(t)\rangle_{L^2}=2\langle \partial_t \xi(t),\Delta \xi(t)\rangle_{L^2}.
    \end{equation*}
We rewrite \eqref{eqB3} in the equivalent form
\begin{equation*}
-\int_0^{2\pi} \partial_t \psi\Delta\psi dx+\int_0^{2\pi}(\Delta\psi)^2dx-\kappa\int_0^{2\pi}\psi^{p+1} \Delta\psi dx=-\int_0^{2\pi}\langle u(t),Q(x)\rangle \psi\Delta\psi dx
\end{equation*}
and, integrating by parts, we obtain
\begin{multline*}
-\left.\partial_t \psi\partial_x\psi\right|_0^{2\pi}+\int_0^{2\pi}\partial_x(\partial_t \psi)\partial_x\psi dx+\int_0^{2\pi}(\Delta\psi)^2dx-\left.\kappa\psi^{p+1}\partial_x\psi \right|_0^{2\pi}+\kappa (p+1)\int_0^{2\pi}\psi^{p}(\partial_x\psi)^2dx\\\leq C(t)\left|\left.\psi\partial_x\psi\right|_0^{2\pi}-\int_0^{2\pi}(\partial_x\psi)^2 dx\right|,
\end{multline*}
and, as above, we get an inequality of the form $
\int_0^{2\pi}\partial_x(\partial_t \psi)\partial_x\psi dx\leq C(t)\int_0^{2\pi}(\partial_x\psi)^2 dx$. This result has been obtained because the first integral on the left-hand side is well-defined since $\psi_t\in L^2((0,T),H^1(\T,\R))$. This is due to the fact that $\psi_0\in H^3(\T,\R)$ and thanks to Proposition \ref{prop-reg-psi-t}. Hence, we deduce that
\begin{equation*}
\frac{1}{2}\partial_t\norm{\partial_x\psi(t)}^2_{L^2}\leq C(t)\norm{\partial_x \psi_0}^2_{L^2},
\end{equation*}
and by the Gr{\"o}nwall's inequality we get
\begin{equation}\label{eq-norm-2-partialx-psi}
\norm{\partial_x\psi(t)}^2_{L^2}\leq\norm{\partial_x \psi_0}^2_{L^2}e^{2\int_0^t C(s)ds}\leq \norm{\partial_x \psi_0}^2_{L^2}e^{2C\sqrt{\TT(\psi_0,u)}\norm{u}_{L^2(0,\TT(\psi_0,u))}}.
\end{equation}
Therefore, by using \eqref{eq-norm-2-psi} and \eqref{eq-norm-2-partialx-psi}, we have proved that
\begin{equation*}
\norm{\psi(t)}^2_{H^1}=\norm{\psi(t)}^2_{L^2}+\norm{\partial_x \psi(t)}^2_{L^2}\leq (\norm{\psi_0}^2_{L^2}+\norm{\partial_x \psi_0}^2_{L^2})e^{2C\sqrt{\TT(\psi_0,u)}\norm{u}_{L^2(0,\TT(\psi_0,u))}},
\end{equation*}
for almost every $t\in(0,\TT(\psi_0,u))$. Since the right-hand side does not depend on $t$, we can take the limit as $t\to \TT(\psi_0,u)$ and conclude that $\norm{\psi(t)}^2_{H^1}\leq C$ as $t\to \TT(\psi_0,u).$
\end{proof}

\section{An estimate for the exact controllability to the ground state solution}\label{AppC}

In this section we derive an estimate for the solution $w$ of \eqref{w} which is useful for the proof of Theorem \ref{teo-loc-nlh}. Recall that $\Phi=1/\sqrt{2\pi}$. We shall take advantage of estimate \eqref{estim-y-yo-v1} that we have obtained for the solution of
\begin{equation}\label{app-eq-y-v1}
    \begin{cases}
        \partial_t y(t)-\partial^2_xy(t)+\kappa\displaystyle\sum_{j=2}^{p+1}\binom{p+1}{j}y^j(t)\Phi^{p+1-j}=\langle v^1(t),Q\rangle y(t)+\langle v^1(t),Q \rangle \Phi,&t\in(s_0,s_1),\\
        y(s_0)=y_{s_0}:=\psi(s_0)-\Phi.
    \end{cases}
\end{equation}

\begin{proposition}\label{app-prop-w}
    Let $\psi(s_0)\in H^3(\T,\R)$ and $v^1\in H^1((s_0,s_1),\R^{q+2})$ with $0\leq s_0<s_1$. Consider $Q\in H^3(\T,\R^{q+2})$ and $y\in C([s_0,s_1],H^3(\T,\R))\cap C^1([s_0,s_1],H^1(\T,\R))$ be the solution of \eqref{app-eq-y-v1}. Let $v^1$ satisfy \eqref{estim-v1}. Then, the solution $w\in C([s_0,s_1],H^1(\T,\R))\cap H^1([s_0,s_1],L^2(\T,\R))$ of 
    \begin{equation}\label{app-w}
        \begin{cases}
           \partial_t w(t)-\partial^2_x w(t)-\kappa(p+1)\Phi^pw(t)+\kappa\displaystyle\sum_{j=2}^{p+1}\binom{p+1}{j}\Phi^{p+1-j} y^j(t)=\langle v^1(t),Q\rangle y(t),&t\in(s_0,s_1),\\
            w(s_0)=0,
        \end{cases}
    \end{equation}
    satisfies the following inequality with $\sigma=s_1-s_0$ and $A_4(\sigma,\norm{y_{s_0}}_{H^1})$ defined in \eqref{app-A4}:
    \begin{equation}\label{app-estim-w}
        \sup_{t\in[s_0,s_1]}\norm{w(t)}_{H^1}\leq A_4(\sigma,\norm{y_{s_0}}_{H^1})\norm{y_{s_0}}^2_{H^1}.
    \end{equation}
    \end{proposition}

\begin{proof}
    We recall that $w$ has been defined as $w:=y-\xi$ with $\xi\in C([s_0,s_1],H^1(\T,\R))\cap H^1(s_0,s_1,L^2(\T,\R))\cap L^2(s_0,s_1,H^2(\T,\R))$ solution of the linear system \eqref{lin-intro} with control $v^1\in H^1_0((s_0,s_1),\R^{q+2})$ such that
    \begin{equation*}
        \xi(s_1;s_0,\xi_0,v^1)=0,\quad\text{with}\quad  \xi(s_0)=y_{s_0}=\psi(s_0)-\Phi,
    \end{equation*}
    and $\norm{v^1}_{H^1(s_0,s_1)}\leq N(\sigma)\norm{y_{s_0}}_{H^1},$ with $\sigma=s_1-s_0. $ The existence of such control has been established in Proposition \ref{prop-null}. Let us estimate the norm of $w$ at time $s_1$. We first multiply the equation in \eqref{app-w} by $w$ and we obtain
\begin{multline*}
     \langle\partial_t w(t),w(t)\rangle_{L^2}-\langle\partial^2_x w(t),w(t)\rangle_{L^2}-\kappa(p+1)\Phi^p\langle w(t),w(t)\rangle_{L^2}+\kappa\displaystyle\sum_{j=2}^{p+1}\binom{p+1}{j}\Phi^{p+1-j} \langle y^j(t),w(t)\rangle_{L^2}\\
     =\big\langle\langle v^1(t),Q\rangle y(t),w(t)\big\rangle_{L^2}.
\end{multline*}
Using the accreativity of $-\partial_x^2$ and that $H^1(\T) \hookrightarrow C(\T)$, we get
\begin{multline*}
     \frac{1}{2}\frac{d}{dt} \norm{w(t)}^2_{L^2}\leq \kappa(p+1)\Phi^p\norm{w(t)}^2_{L^2}+\\+ \kappa\displaystyle\sum_{j=2}^{p+1}\binom{p+1}{j}\Phi^{p+1-j} \left(\frac{\norm{y(t)}^{2j}_{H^1}}{2}+\frac{\norm{w(t)}_{L^2}^2}{2}\right)+C_Q^2\sum_{j=1}^{q+2}|v^1_j(t)|^2\frac{\norm{y(t)}_{L^2}^2}{2}+\frac{\norm{w(t)}_{L^2}^2}{2},
\end{multline*}
where
\begin{equation}\label{c_Q}
    C_Q:=\sup_{i=1,\dots,{q+2}}\norm{Q^i}_{C^0}.
\end{equation}
By the Gr{\"o}nwall's Lemma, for any $t\in(s_0,s_1)$ it holds that
\begin{multline*}
    \norm{w(t)}_{L^2}^2\leq \left(\int_{s_0}^t \kappa\sum_{j=2}^{p+1}\binom{p+1}{j}\Phi^{p+1-j}\norm{y(s)}^{2j}_{H^1}+C_Q^2\sum_{j=1}^{q+2} |v^1_j(s)|^2\norm{y(s)}_{H^1}^2ds \right)\cdot \\e^{\int_{s_0}^t \left(2\kappa(p+1)\Phi^p+\sum_{j=2}^{p+1}\binom{p+1}{j}\Phi^{p+1-j}+1\right)ds}.
\end{multline*}
Therefore, by taking the supremum over $[s_0,s_1]$ we obtain
\begin{equation}\label{app-estim-norm-w}
    \begin{split}
        &\sup_{t\in[s_0,s_1]}\norm{w(t)}^2_{L^2} \leq \left(\kappa \sigma\sum_{j=2}^{p+1}\binom{p+1}{j}\Phi^{p+1-j}\sup_{t\in[s_0,s_1]}\norm{y(t)}^{2j}_{H^1}+C_Q^2\norm{v^1}^2_{L^2(s_0,s_1)}\sup_{t\in[\tau_0,\tau_1]}\norm{y(t)}_{H^1}^2 \right) e^{A_1(\sigma)}\\
        &\leq \left(2\kappa\sigma(p+1)\sum_{j=2}^{p+1}\binom{p+1}{j}\Phi^{p+1-j}\left(\norm{y_{s_0}}^{2j}_{H^1}+\norm{v^1}^{2j}_{L^2(s_0,s_1)}\right)+C_Q^2\norm{v^1}^2_{L^2(s_0,s_1)}\left(\norm{y_{s_0}}^2_{H^1}+\norm{v^1}^2_{L^2(s_0,s_1)} \right)\right) e^{A_1(\sigma)}\\
        &\leq \left(2\kappa\sigma(p+1)\sum_{j=2}^{p+1}\binom{p+1}{j}\Phi^{p+1-j}\left(1+N(\sigma)^{2j}\right)\norm{y_{s_0}}^{2j}_{H^1}+C_Q^2N(\sigma)^2\left(1+N(\sigma)^2\right)\norm{y_{s_0}}^{4}_{H^1} \right) e^{A_1(\sigma)}\\
        &\leq \left(2\kappa\sigma(p+1)\norm{y_{s_0}}^{4}_{H^1}\sum_{j=2}^{p+1}\binom{p+1}{j}\Phi^{p+1-j}\left(1+N(\sigma)^{2j}\right)\norm{y_{s_0}}^{2(j-2)}_{H^1}+C_Q^2N(\sigma)^2\left(1+N(\sigma)^2\right)\norm{y_{s_0}}^{4}_{H^1} \right) e^{A_1(\sigma)}\\
        &\leq A_2(\sigma,\norm{y_{s_0}}_{H^1})^2\norm{y_{s_0}}^{4}_{H^1},
    \end{split}
\end{equation}
where
\begin{equation*}
    A_1(\sigma):=\sigma\left(2\kappa(p+1)\Phi^p+\kappa\sum_{j=2}^{p+1}\binom{p+1}{j}\Phi^{p+1-j}+1\right),
\end{equation*}
\begin{multline*}
    A_2(\sigma,\norm{y_{s_0}}_{H^1}):=\\
    \left(2\kappa\sigma(p+1)\sum_{j=2}^{p+1}\binom{p+1}{j}\Phi^{p+1-j}\left(1+N(\sigma)^{2j}\right)\norm{y_{s_0}}^{2(j-2)}_{H^1}+C_Q^2N(\sigma)^2\left(1+N(\sigma)^2\right)\right)^{1/2}e^{A_1(\sigma)/2}.
\end{multline*}
Let us now multiply the equation in \eqref{w} by $-\partial^2_x w(t)$
\begin{multline*}
     -\langle\partial_t w(t),\partial^2_x w(t)\rangle_{L^2}+\langle\partial^2_x w(t),\partial^2_x w(t)\rangle_{L^2}+\kappa(p+1)\Phi^p\langle w(t),\partial^2_x w(t)\rangle_{L^2}\\
     -\kappa\displaystyle\sum_{j=2}^{p+1}\binom{p+1}{j}\Phi^{p+1-j} \langle y^j(t),\partial^2_x w(t)\rangle_{L^2}=-\big\langle\langle v^1(t),Q\rangle y(t),\partial^2_x w(t)\big\rangle_{L^2}.
\end{multline*}
We now perform integrations by parts, and we get
\begin{multline*}
     \langle\partial_t (\partial_xw(t)),\partial_x w(t)\rangle_{L^2}+\norm{\partial^2_x w(t)}^2_{L^2}-\kappa(p+1)\Phi^p\langle \partial_x w(t),\partial_x w(t)\rangle_{L^2}\\
     +\kappa\displaystyle\sum_{j=2}^{p+1}\binom{p+1}{j}\Phi^{p+1-j} j\langle y^{j-1}(t)\partial_x y(t),\partial_x w(t)\rangle_{L^2}
     =-\langle\langle v^1(t),Q\rangle y(t),\partial^2_x w(t)\rangle_{L^2},
\end{multline*}
and therefore
\begin{multline*}
     \frac{1}{2}\frac{d}{dt}\norm{\partial_x w(t)}^2_{L^2}+\norm{\partial^2_x w(t)}^2_{L^2}\leq
         \kappa(p+1)\Phi^p \norm{\partial_x w(t)}_{L^2}^2\\
         +\kappa(p+1)\displaystyle\sum_{j=2}^{p+1}\binom{p+1}{j}\Phi^{p+1-j} \left(\frac{\norm{y(t)}^{2j}_{H^1}}{2}+\frac{\norm{\partial_x w(t)}_{L^2}^2}{2}\right)   +C_Q^2\sum_{j=1}^{q+2}|v^{1,j}(t)|^2\norm{y(t)}^2_{H^1}+\frac{\norm{\partial^2_x w(t)}_{L^2}^2}{2},
\end{multline*}
where we have used that $\partial_t w=\partial_t y-\partial_t\xi \in L^2((s_0,s_1),H^1(\T,\R))$ (see Proposition \ref{prop-reg-psi-t} for $y$ and Remark \ref{reg-sol-xi} for $\xi$). We apply the Gr{\"o}nwall's Lemma, and we obtain that, for any $t\in[s_0,s_1]$,
\begin{equation*}
     \norm{\partial_x w(t)}_{L^2}^2\leq\left(\int_{s_0}^t \kappa(p+1)\sum_{j=2}^{p+1}\binom{p+1}{j}\Phi^{p+1-j} \norm{y(s)}^{2j}_{H^1} +C^2_Q\sum_{j=1}^{q+2}|v^{1,j}(s)|^2\norm{y(s)}^2_{H^1}ds\right)e^{A_3(s_0,t)}
\end{equation*}
with
\begin{equation*}
    A_3(s_0,t):=\int_{s_0}^t\Big(2\kappa(p+1)\Phi^p+\kappa(p+1)\sum_{j=2}^{p+1}\binom{p+1}{j}\Phi^{p+1-j}\Big)ds.
\end{equation*}
Taking the supremum over the interval $[s_0,s_1]$, we obtain
\begin{equation*}
    \begin{split}
         &\sup_{t\in[s_0,s_1]}\norm{\partial_x w(t)}^2_{L^2}\leq\left(\int_{s_0}^{s_1} \kappa(p+1)\sum_{j=2}^{p+1}\binom{p+1}{j}\Phi^{p+1-j} \norm{y(s)}^{2j}_{H^1} +C^2_Q\sum_{j=1}^{q+2}|v^{1,j}(s)|^2\norm{y(s)}^2_{H^1}ds\right)e^{A_3(s_0,s_1)}\\
         &\quad\leq \left(\kappa\sigma(p+1) \sum_{j=2}^{p+1}\binom{p+1}{j}\Phi^{p+1-j} \sup_{t\in[s_0,s_1]}\norm{y(t)}^{2j}_{H^1} +C^2_Q\norm{v^1}^2_{L^2(s_0,s_1)}\sup_{t\in[s_0,s_1]}\norm{y(t)}^2_{H^1}\right)e^{A_3(s_0,s_1)}\\
         &\quad\leq \left(2\kappa\sigma(p+1)^2 \sum_{j=2}^{p+1}\binom{p+1}{j}\Phi^{p+1-j} \left(\norm{y_{s_0}}^{2j}_{H^1}+\norm{v^1}^{2j}_{L^2(s_0,s_1)}\right) +C^2_Q\norm{v^1}^2_{L^2(s_0,s_1)}\left(\norm{y_{s_0}}^2_{H^1}+\norm{v^1}^2_{L^2(s_0,s_1)}\right)\right)e^{A_3(s_0,s_1)}.
     \end{split}
\end{equation*}
Thanks to the assumption on $y_0$ we deduce
\begin{equation}\label{app-estim-norm-dw}
    \begin{split}
         \sup_{t\in[s_0,s_1]}&\norm{\partial_x w(t)}^2_{L^2}\\
         &\leq\left(2\kappa\sigma(p+1)^2 \sum_{j=2}^{p+1}\binom{p+1}{j}\Phi^{p+1-j} \left(1+N^{2j}(\sigma)\right)\norm{y_{s_0}}^{2j}_{H^1} +C_Q^2N^2(\sigma)\left(1+N^2(\sigma)\right)\norm{y_{s_0}}^4_{H^1}\right)e^{A_3(s_0,s_1)}\\
         &=\left(2\kappa\sigma(p+1)^2 \sum_{j=2}^{p+1}\binom{p+1}{j}\Phi^{p+1-j} \left(1+N^{2j}(\sigma)\right)\norm{y_{s_0}}^{2(j-2)}_{H^1} +C^2_Q N^2(\sigma)\left(1+N^2(\sigma)\right)\right)e^{A_3(s_0,s_1)}\norm{y_{s_0}}^4_{H^1}\\
         &=A_3(\sigma,\norm{y_{s_0}}_{H^1})^2\norm{y_{s_0}}^4_{H^1},
     \end{split}
\end{equation}
where
\begin{multline*}
    A_3(\sigma,\norm{y_{s_0}}_{H^1}):=\\
    \left(2\kappa\sigma(p+1)^2 \sum_{j=2}^{p+1}\binom{p+1}{j}\Phi^{p+1-j} \left(1+N^{2j}(\sigma)\right)\norm{y_{s_0}}^{2(j-2)}_{H^1} +C^2_Q N^2(\sigma)\left(1+N^2(\sigma)\right)\right)^{1/2}e^{A_3(s_0,s_1)/2}.
\end{multline*}
Finally, from \eqref{app-estim-norm-w} and \eqref{app-estim-norm-dw}, we conclude that
\begin{equation*}
    \begin{split}
        \sup_{t\in[s_0,s_1]}\norm{w(t)}^2_{H^1}&\leq 2\left(\sup_{t\in[s_0,s_1]}\norm{w(t)}^2_{L^2}+\sup_{t\in[s_0,s_1]}\norm{\partial_x w(t)}^2_{L^2}\right)\\
        &\leq 2\left(A_2(\sigma,\norm{y_{s_0}}_{H^1})^2+A_3(\sigma,\norm{y_{s_0}}_{H^1})^2\right)\norm{y_{s_0}}^4_{H^1}\leq A_4^2(\sigma,\norm{y_{s_0}}_{H^1})^2 \norm{y_{s_0}}^4_{H^1},
    \end{split}
\end{equation*}
with
\begin{multline}\label{app-A4}
    A_4(\sigma,\norm{y_{s_0}}_{H^1}):=    \sqrt{2}\left(2\kappa\sigma(p+1)^2 \sum_{j=2}^{p+1}\binom{p+1}{j}\Phi^{p+1-j} \left(1+N(\sigma)^{2j}\right)\norm{y_{s_0}}^{2(j-2)}_{H^1} +C^2_Q N^2(\sigma)\left(1+N(\sigma)^2\right)\right)^{1/2}\cdot\\
    e^{\sigma(2\kappa(p+1)\Phi^p+\kappa(p+1)\sum_{j=2}^{p+1}\binom{p+1}{j}\Phi^{p+1-j}+1))/2}.
\end{multline}
The proof is therefore concluded.
\end{proof}

\end{appendices}

\end{document}